\documentclass[reqno,dvipsnames]{amsart}
\pdfoutput=1

\usepackage{amssymb}
\usepackage{amsmath}
\usepackage{mathabx}
\usepackage{hyperref}
\usepackage{graphicx}
\usepackage{mathrsfs}

\usepackage{epigraph}   

\usepackage{amssymb, amsmath}
\usepackage{mathrsfs}
\usepackage{amscd} 
\usepackage{amsmath}
\usepackage[english]{babel}
\usepackage{microtype}               
\usepackage[T1]{fontenc}
\usepackage{lmodern}

\usepackage{hyperref}
\hypersetup{
	colorlinks   = true,          
	urlcolor     = blue,          
	linkcolor    = blue,          
	citecolor   = red             
}

\usepackage{enumitem}
\setenumerate[1]{label=(\roman*)}    

\usepackage[normalem]{ulem}
\usepackage{graphicx}
\usepackage[all,cmtip]{xy}
\usepackage{tikz}
\usetikzlibrary{arrows,decorations,shapes,shadows}
\usetikzlibrary{calc,intersections}
\usepackage{verbatim}

\usepackage{epigraph}   

\newbox\mybox
\def\overtag#1#2#3{\setbox\mybox\hbox{$#1$}\hbox to
  0pt{\vbox to 0pt{\vglue-#3\vglue-\ht\mybox\hbox to \wd\mybox
      {\hss$\ss#2$\hss}\vss}\hss}\box\mybox}
\def\undertag#1#2#3{\setbox\mybox\hbox{$#1$}\hbox to 0pt{\vbox to
    0pt{\vglue#3\vglue\ht\mybox\hbox to \wd\mybox
      {\hss$\ss#2$\hss}\vss}\hss}\box\mybox}
\def\lefttag#1#2#3{\hbox to 0pt{\vbox to 0pt{\vglue -6pt\hbox to
      0pt{\hss$\ss#2$\hskip#3}\vss}}#1}
\def\righttag#1#2#3{\hbox to 0pt{\vbox to 0pt{\vglue -6pt\hbox to
      0pt{\hskip#3$\ss#2$\hss}\vss}}#1}
\let\ss\scriptstyle
\def\splicediag#1#2{\xymatrix@R=#1pt@C=#2pt@M=0pt@W=0pt@H=0pt}
\def\Dot{\lower.2pc\hbox to 2pt{\hss$\bullet$\hss}}
\def\Circ{\lower.2pc\hbox to 2pt{\hss$\circ$\hss}}
\def\Vdots{\raise5pt\hbox{$\vdots$}}
\newcommand\lineto{\ar@{-}}
\newcommand\dashto{\ar@{--}}
\newcommand\dotto{\ar@{.}}

\let\cal\mathcal
\renewcommand{\setminus}{\smallsetminus}
\newcommand\Q{{\mathbb Q}}

\newcommand\C{{\mathbb C}}
\newcommand\Z{{\mathbb Z}}
\newcommand\N{{\mathbb N}}

\newcommand\cdiamond{{\bigskip\begin{center}$\diamond$\end{center}\bigskip}}

\DeclareMathOperator{\val}{val}
\DeclareMathOperator{\ord}{ord}

\DeclareMathOperator{\inn}{inn}

\DeclareMathOperator{\Aux}{Aux}

\newtheorem{theorem}{Theorem}[section]
\newtheorem{proposition}[theorem]{Proposition}
\newtheorem*{theorem*}{Theorem}
\newtheorem{corollary}[theorem]{Corollary}
\newtheorem{lemma}[theorem]{Lemma}

\theoremstyle{definition}

\newtheorem*{amalgamation*}{Amalgamation}
\newtheorem{example}[theorem]{Example}
\newtheorem{remark}[theorem]{Remark}

\newtheorem*{remark*}{Remark}
\newtheorem{definition}[theorem]{Definition}

\theoremstyle{plain}
\newcounter{algorithm}

\renewcommand{\int}{\operatorname{int}}

\usepackage{xcolor}



\newcommand\pa{{\mathfrak a}}

\newcommand\pb{{\mathfrak b}}

\begin{document}

	\title[Polar exploration of complex surface germs]{Polar exploration of complex surface germs}

\author[A. Belotto da Silva]{Andr\'e Belotto da Silva}
\address{Aix-Marseille Universit\'e, CNRS, Centrale Marseille, I2M, Marseille, France}
\email{\href{mailto:andre-ricardo.BELOTTO-DA-SILVA@univ-amu.fr}{andre-ricardo.belotto-da-silva@univ-amu.fr}}
\urladdr{\url{https://andrebelotto.com}}

\author[L. Fantini]{Lorenzo Fantini}
\address{Goethe-Universit\"at Frankfurt, Institut f\"ur Mathematik, Frankfurt am Main, Germany}
\email{\href{mailto:fantini@math.uni-frankfurt.de}{fantini@math.uni-frankfurt.de}}
\urladdr{\url{https://lorenzofantini.eu/}}

\author[A. N\'emethi]{Andr\'as N\'emethi}
\address{Alfr\'ed R\'enyi Institute of Mathematics, ELKH\newline \hspace*{4mm}
	Re\'altanoda utca 13-15, H-1053, Budapest, Hungary \newline \hspace*{4mm}
	ELTE - University of Budapest, Dept. of Geometry, Budapest, Hungary \newline
	\hspace*{4mm}
	BCAM - Basque Center for Applied Math.,
	Mazarredo, 14 E48009 Bilbao, Basque Country, Spain}
\email{\href{mailto:nemethi.andras@renyi.hu}{nemethi.andras@renyi.hu}}
\urladdr{\url{https://users.renyi.hu/~nemethi/}}

\author[A. Pichon]{Anne Pichon}
\address{Aix-Marseille Universit\'e, CNRS, Centrale Marseille, I2M, Marseille, France}
\email{\href{mailto:anne.pichon@univ-amu.fr}{anne.pichon@univ-amu.fr}}
\urladdr{\url{http://iml.univ-mrs.fr/~pichon/}}

\begin{abstract}
We prove that the topological type of a normal surface singularity $(X,0)$ provides finite bounds for the multiplicity and polar multiplicity  of $(X,0)$, as well as for the combinatorics of the families of generic hyperplane sections and of polar curves of the generic plane projections of $(X,0)$.
A key ingredient in our proof is a topological bound of the growth of the Mather discrepancies of $(X,0)$, which allows us to bound the number of point blowups necessary to achieve factorization of any resolution of $(X,0)$ through its Nash transform.
This fits in the program of \emph{polar explorations}, the quest to determine the generic polar variety of a singular surface germ, to which the final part of the paper is devoted.
\end{abstract}

\maketitle

\setlength{\epigraphwidth}{9.3truecm}
\epigraph{A ranger walked from his tent 10 km southwards, turned east, walked
	straight eastwards 10 km more, met his bear friend, turned north and after another 10 km found himself by his tent. 
	What colour was the bear and where did all this happen?
}{V. I. Arnold's (Odessa, 12 June 1937 -- Paris, 3 June 2010)\\
selection of problems for children from 5 to 15}

\section{Introduction}

{\renewcommand*{\thetheorem}{\Alph{theorem}}

A normal complex surface singularity $(X,0)$ can be resolved either by a sequence of normalized point blowups, following seminal work of Zariski \cite{Zariski1939} from the late nineteen thirties, or by a sequence of normalized Nash transforms, as was done half a century later by Spivakovsky \cite{Spivakovsky1990}.
The main goal of this paper is to shed some light on the relationship between these two resolution algorithms, which despite their importance and their centrality in modern mathematics is still quite mysterious, providing some evidence of a duality between the two which was initially observed by L\^e \cite[\S4.3]{Le2000}. 

While the blowup $\mathrm{Bl}_0X$ of  the maximal ideal of $(X,0)$ is the minimal transformation which resolves the family of generic hyperplane sections of $(X,0)$, the Nash transform  $\nu$ of $(X,0)$ is the minimal transformation that resolves the family of the polar curves associated with the generic plane projections of $(X,0)$.
Therefore, the study of the duality of resolution algorithms translates into the study of the relative positions on $(X,0)$ of those two families of curves.
This is the viewpoint we adopt in this paper. Our main theorem roughly states that fixed the topology of $(X,0)$, that is the homeomorphism class of its link, there are, up to homeomorphism, only a finite number of possible relative positions between these families of curves.

In order to give a precise statement of our result we need to introduce some additional notation. 
Let $\pi\colon X_\pi \to X$ be a good resolution of $(X,0)$, by which we mean a proper bimeromorphic
morphism from a smooth surface $X_\pi$ to $X$ which is an isomorphism outside of a simple normal crossing divisor $E=\pi^{-1}(0)$, and denote by $V(\Gamma_\pi)$ the set of vertices of the dual graph $\Gamma_\pi$ of $\pi$, so that every element $v$ of $V(\Gamma_\pi)$ corresponds to an irreducible component $E_v$ of $E$. 
We weight $\Gamma_\pi$ by attaching to each vertex $v$ the genus $g(v)\geq 0$ of the complex curve $E_v$ and the self-intersection $e(v)<0$ of $E_v$. 

For each $v$ in $V(\Gamma_\pi)$, we denote by $l_v$ the intersection multiplicity of the zero locus $h^{-1}(0)$ of a generic hyperplane section $h\colon(X,0)\to(\C,0)$ with $E_v$, and we call \emph{$\cal L$-vector} of $(X,0)$ the vector $L_\pi=(l_v)_{v\in V(\Gamma_\pi)}\in\Z_{\geq0}^{V(\Gamma_\pi)}$. 
Whenever $\pi\colon X_\pi \to X$ factors through $\mathrm{Bl}_0X$, the strict transform of such a generic hyperplane section via $\pi$ consists of a disjoint union of smooth curves that intersect transversely $E$ at smooth points of $E$, and $l_v$ is the number of such curves passing through the component $E_v$. Similarly, we denote by $p_v$ the intersection multiplicity of the strict transform of the polar curve of a generic plane projection $\ell\colon(X,0)\to(\C^2,0)$ with  $E_v$ and  we call \emph{$\cal P$-vector} of $(X,0)$ the vector $P_\pi=(p_v)_{v\in V(\Gamma_\pi)}\in\Z_{\geq0}^{V(\Gamma_\pi)}$. Whenever $\pi\colon X_\pi \to X$ factors through $\nu$ then such a strict transform consists of smooth curves intersecting $E$ transversely at smooth points, and $p_v$ equals the number of such curves through $E_v$. We can now give a precise statement of our main result:

\begin{theorem} \label{thm:main}
Let $M$ be a real $3$-manifold. 
There exists finitely many triplets $(\Gamma,L,P)$, where $\Gamma$ is a weighted graph and $L$ and $P$ are vectors in $(\Z_{\geq0})^{V(\Gamma)}$, such that there exists a normal surface singularity $(X,0)$ satisfying the following conditions:
\begin{enumerate}
\item The link of $(X,0)$ is homeomorphic to $M$.
\item $(\Gamma,L,P) = (\Gamma_\pi, L_\pi, P_\pi)$, where $\pi \colon X_{\pi} \to X$ is the minimal good resolution of $(X,0)$ which factors through the blowup of the maximal ideal and the Nash transform of $(X,0)$.
\end{enumerate}
\end{theorem}

Recall that the \emph{link} of a normal surface singularity $(X,0)$, which is defined by embedding $(X,0)$ in a suitable smooth germ $(\mathbb{C}^N,0)$ and intersecting it with a small sphere, is, up to homeomorphism, a well defined real 3-manifold which determines and is determined by the homeomorphism class of the germ $(X,0)$ thanks to the Conical Structure Theorem.
Equivalently, the topological type of $(X,0)$ can be completely described in terms of the weighted dual graph $\Gamma_{\pi}$ of any good resolution $\pi$ of $(X,0)$, since $\Gamma_{\pi}$ is a plumbing graph of the link of $(X,0)$.
Conversely, Neumann \cite{Neumann1981} proved that the weighted dual graph of the minimal good resolution of $(X,0)$ is determined by the topology of the surface germ.
Adopting this point of view, the datum of the 3-manifold $M$ of Theorem~\ref{thm:main} is equivalent to the one of a weighted dual graph $\Gamma$, and as a consequence of the theorem we obtain the following:

\begin{corollary}\label{cor:main}
	Let $\Gamma$ be a weighted graph.
	Then there exist finitely many pairs $(L,P)$ of vectors $L$ and $P$ in $(\Z_{\geq0})^{V(\Gamma)}$ such that there exist a normal surface singularity $(X,0)$ and a good resolution $\pi$ of $(X,0)$ satisfying
	\[
	(\Gamma,L,P)=(\Gamma_{\pi}, L_{\pi}, P_{\pi}).
	\]
\end{corollary}

One of the ingredients of the proof of  Theorem~\ref{thm:main}  is the fact that the topological type of a normal surface singularity gives a bound of the multiplicity  of the germs realizing it.
We believe this statement to be of independent interest:

\begin{proposition}\label{prop:main}
	Let $M$ be a real $3$-manifold.
	Then there exists a natural number $n_M$ that only depends on the homeomorphism type of $M$ and such that, if $(X,0)$ is a normal surface singularity whose link is homeomorphic to $M$, the multiplicity $m(X,0)$ of $(X,0)$ is at most $n_M$.
\end{proposition}
  
Moreover, an explicit value for the bound $n_M$ can be computed in terms of the topology of $M$.
Our proof of this result makes use of a construction of Caubel, Popescu-Pampu, and the third author \cite{CaubelNemethiPopescu-Pampu2006}.

Given a weighted graph $\Gamma$, Proposition~\ref{prop:main} would then be sufficient to prove the finiteness of the set of the $\cal L$-vectors $L$ such that the pair $(\Gamma,L)$ can be realized by a surface singularity $(X,0)$.
By a procedure that we call \emph{gardening}, we then obtain the finiteness of pairs  $(\Gamma,L)$ such that $\Gamma$ is the graph of the minimal good resolution factoring through the blowup of the maximal ideal. 
In order to obtain the finiteness of the $P$-vector, we then use the well-known L\^e--Greuel--Teissier formula \cite{LeTeissier1981} to deduce from Proposition~\ref{prop:main} a bound on the multiplicity of the polar curve of $(X,0)$ in terms of $n_M$ and of the Euler characteristic of the Milnor--L\^e fiber of a generic linear form on $(X,0)$, which can be computed in terms of the graph $\Gamma$.

While the argument above would suffice to deduce Corollary~\ref{cor:main}, in order to prove Theorem~\ref{thm:main} we also need to prove that the topological type of $(X,0)$ provide a bound of the number of point blowups necessary to go from any good resolution of $(X,0)$ to one factoring through the Nash transform of $(X,0)$.
We do this by considering a set of invariants, the so-called \emph{Mather discrepancies} introduced by de Fernex, Ein, and Ishii \cite{FernexEinIshii2008}, and proving that they are bounded from above by another invariant $\nu_v$ which only depends on the topological type of $(X,0)$.
We conclude by showing that the Mather discrepancies grow faster than the $\nu_v$ do when we perform any blowup necessary to achieve factorization through the Nash transform, which permits us to set up an inductive argument.
The key technical result allowing us to do this is Theorem~\ref{thm:sheaf_2-forms}, which proves the existence of a suitable sheaf of K\"ahler 2-forms that only depends on the topological type of $(X,0)$, leading to the definition of the invariants $\nu_v$.

\cdiamond

While Theorem~\ref{thm:main} provides a finite list of possibly realizable pairs of $\cal L$- and $\cal P$-vectors $L$ and $P$, the list outputted by its proof could still be fairly long.
In the final section of the paper we discuss how to sharpen this bound by studying additional restrictions on the relative positions of generic hyperplane sections and polar curves.

We recast this problem in the framework of \emph{polar exploration}, which is the quest of determining the $\cal P$-vector $P$ which can be realized by a normal surface singularity which realizes a fixed pair $(\Gamma,L)$.

In order to approach this question, we build on the so-called \emph{Laplacian formula} of a normal surface singularity.
This result, proven by three of the authors of the present paper in \cite{BelottodaSilvaFantiniPichon2019}, can be thought of as a local version of the L\^e--Greuel--Teissier formula referred to above.
It describes the behavior of an infinite family of metric invariants of a normal surface singularity $(X,0)$, called its \emph{inner rates}, that appeared naturally in the study of the Lipschitz geometry of $(X,0)$ in the foundational work \cite{BirbrairNeumannPichon2014}.
This tool has been used in the previous work \cite[Theorem~1.1]{BelottodaSilvaFantiniPichon2020} to prove that the problem of polar exploration admits a unique solution for a specific class of surface singularities, those that are \emph{Lipschitz Normally Embedded}.

Additional restrictions on the relative hyperplane and polar positions can be derived from the topology of Milnor--L\^e fibers; this is discussed in Lemma~\ref{lem:hurwitz}.

We conclude the paper by discussing in detail an example from \cite{MaugendreMichel2017}, for which by combining the Laplacian formula with the topological constraints from Lemma~\ref{lem:hurwitz} we obtain a unique solution to the problem of polar exploration (see Example~\ref{ex:MaugendreMichel2017}).

\subsection*{Acknowledgments}
We would like to thank Hussein Mourtada and Ana Reguera for useful discussions. 
This work has been partially supported by the project \emph{Lipschitz geometry of singularities (LISA)} of the \emph{Agence Nationale de la Recherche} (project ANR-17-CE40-0023).
The second author has also been partially supported by a \emph{Research Fellowship} of the \emph{Alexander von Humboldt Foundation}, while the third author has been partially supported by the NKFIH Grant “Élvonal (Frontier)” KKP 126683.
}


\section{Preliminaries on Lipman cones}

In this section we begin by recalling the notion of Lipman cone, and then prove an adaptation of a result of Caubel, Popescu-Pampu, and the third author from \cite{CaubelNemethiPopescu-Pampu2006} which will be useful in the remaining part of the paper.
A more thorough discussion of the basic objects described in this section can be found in \cite{Nemethi1999}.

\medskip

Let $\Gamma$ be a finite connected graph without loops and such that each vertex $v\in V(\Gamma)$ is weighted by two integers $g(v)\geq0$, called genus, and $e(v)\leq -1$, called self-intersection.
We assume that the incidence matrix induced by the self-intersections of the vertices of $\Gamma$, that is the matrix $I_\Gamma\in \Z^{V(\Gamma)}$ whose $(v,v')$-th entry is $e(v)$ if $v=v'$, and the number of edges of $\Gamma$ connecting $v$ to $v'$ otherwise, is negative definite.	
Let $E=\bigcup_{v\in V(\Gamma)}E_v$ be a configuration of curves whose dual graph is $\Gamma$, so that $I_{\Gamma} = (E_v \cdot E_{v'})$, and consider the free additive group $\cal G$ generated by the irreducible components of $E$, that is 
\[
\cal G = \bigg\{D =\sum_{v \in V(\Gamma)} d_v E_v \,\bigg|\, d_v \in \Z\bigg\}. 
\]
By a slight abuse of notation, we refer to the elements of $\cal G$ as \emph{divisors on $\Gamma$}.
On $\cal G$ there is a natural intersection pairing $D\cdot D'$, described by the incidence matrix $I_{\Gamma}$, and a natural partial ordering given by setting $\sum d_v E_v \leq \sum d'_v E_v$ if and only if $d_v \leq d'_v$ for every $v \in V(\Gamma)$. 

The {\it Lipman cone} of $\Gamma$ is the semi-group $\cal E^+$ of $\cal G$  defined as 
\[
\cal E^+ = \big\{D \in \cal G \,\big|\, D \cdot E_v \leq 0  \text{ for all } v \in V(\Gamma) \big\}.
\]

\begin{remark}
	By looking at the coefficients of a divisor we can identify $\cal G$ with the additive group $\Z^{V(\Gamma)}$.
	Then the Lipman cone $\cal E^+$ of $\Gamma$ is naturally identified with the cone $\Z_{\geq0}^{V(\Gamma)}\cap-I_\Gamma^{-1}\big(\Q_{\geq0}^{V(\Gamma)}\big)$, since by definition a divisor $\sum d_vE_v$ belongs to $\cal E^+$ if and only if the vector $I_\Gamma\cdot(d_v)_{v\in V(\Gamma)}$ belongs to $\Z_{\leq0}^{V(\Gamma)}$.
\end{remark}

A cardinal property of the Lipman cone $\cal E^+$, proven in \cite[Proposition 2]{Artin1966}, is that it has a unique nonzero minimal element $Z_{\min}^{\Gamma}$, called the {\it fundamental cycle} of $\Gamma$, and that moreover $Z_{\min}^{\Gamma} \succ 0$, that is the coefficients of $Z_{\min}^{\Gamma}$ are all strictly positive.
Observe that the existence of the fundamental cycle and the fact that $Z_{\min}^{\Gamma}\succ 0$ are equivalent to the fact that $D\succ 0$ for every nonzero divisor $D$ in $\cal E^+$.

Assume from now on that $\Gamma$ is the dual graph of a good resolution $\pi$ of a normal surface singularity $(X,0)$.
Notice that the Lipman cone, and therefore its fundamental cycle, only depend on the graph $\Gamma$, that is on the topology of $(X,0)$, and not on the complex geometry of $(X,0)$; the fundamental cycle $Z_{\min}^{\Gamma}$ can be computed from $\Gamma$ by using Laufer's algorithm from \cite[Proposition 4.1]{Laufer1972}.

Consider now a germ of analytic function $f \colon (X, 0) \to (\C, 0)$. 
The \emph{total transform} of $f$ by $\pi$  is the divisor $(f) = (f)_{\Gamma} + f^*$ on $X_\pi$, where $f^*$ is the strict transform of $f$ and $(f)_{\Gamma} = \sum_{v \in V(\Gamma)} m_v(f) E_v$ is the divisor supported on $E$ such that $m_v(f)$ is the multiplicity of $f \circ \pi$ along $E_v$. 
By \cite[Theorem 2.6]{Laufer1971}, we have
\begin{equation}\label{eq:IdentityTotalTransform}
(f) \cdot E_v=0 \quad \text{for all }v \in V(\Gamma).
\end{equation}
In particular, $(f)_\Gamma$ belongs to the Lipman cone $\cal E^+$ of $\Gamma$, and therefore the semi-group $\cal A_X^+ = \{ (f)_{\Gamma} \;|\; f \in \cal O_{X,0}\}$\label{def:A+} of $\cal G$ is contained in $\cal E^+$; it has a unique nonzero minimal element $Z_{\max}^{\Gamma}(X,0)$, which is called the {\it maximal ideal divisor} of $(X,0)$. 
%
%
Observe that $Z_{\max}^{\Gamma}(X,0)$ coincides with the cycle $(h)_{\Gamma}$ of a generic linear form $h \colon (X,0) \to (\C,0)$, so that $l_v = - Z_{\max}^{\Gamma}(X,0) \cdot E_v$ for all $v$ in $V(\Gamma)$, and, by the definition of the fundamental cycle, $Z_{\min}^{\Gamma} \leq  Z_{\max}^{\Gamma}(X,0)$.

In general the inclusion $\cal A_X^+ \subset \cal E^+$ is strict and thus we may have  $Z_{\min}^{\Gamma} \neq  Z_{\max}^{\Gamma}(X,0)$.
However, given a weighted graph $\Gamma$ with negative definite intersection matrix, for all $D \in  \cal E^+$, there exists  a normal complex surface singularity $(X,0)$ and a resolution $\pi \colon X_{\pi} \to X$ such that $\Gamma_{\pi}=\Gamma$ and $D \in \cal A_X^+$ (\cite{Pichon2001} or \cite{NemethiNeumannPichon2011}). 
So in particular, there exists  $(X,0)$  such that $Z_{\min}^{\Gamma} = Z_{\max}^{\Gamma}$. However, when $D$ is sufficiently big, it can be obtained as cycle of an analytic function on any surface singularity realizing the weighted graph $\Gamma$, as showed in \cite[Theorem~4.1]{CaubelNemethiPopescu-Pampu2006}. 
The following Proposition is an adaptation of that result.

\begin{proposition}
\label{thm:CaubelNemethiPopescu-Pampu2006}
	Let $\Gamma$ be a weighted graph and let $D\geq0$ be a nonzero effective divisor on $\Gamma$.
	Then:
	\begin{enumerate}\label{inequality_intersection_divisor}
		\item Assume that for every vertex $v$ of $\Gamma$ we have
	\begin{equation}\label{eq:Ineq1}
	D \cdot E_v + \val_\Gamma(v) + 2g(v) \leq 0,
	\end{equation}
	where $\val_\Gamma(v)$ denotes the valency of $v$ in $\Gamma$.
	Then, for every normal surface singularity $(X,0)$ and every good resolution $\pi \colon (X_{\pi},E) \to (X,0)$ of $(X,0)$ whose weighted dual graph is $\Gamma$, there exists a function $f \in \mathfrak M_{X,0}$ with an isolated singularity at $0$ such that $(f)=(f)_{\Gamma} + f^{\ast}$ is a normal crossing divisor on $X_{\pi}$ and $(f)_{\Gamma} = D$.
	Moreover, the line bundle $\mathcal O_{X_\pi}(-D)$ has no basepoints (that is, for every point $\pa\in E$ there exists a global section $s\in H^{0}\big(X_{\pi},\mathcal{O}_{X_{\pi}}(-D)\big)$ such that $s(\pa) \neq 0$).
	\item If moreover the stronger inequality
	\begin{equation}\label{eq:Ineq2}
	D \cdot E_v + \val_\Gamma(v) + 2g(v) +2 \leq 0
	\end{equation}
	holds for every vertex $v$ of $\Gamma$, then for every free point $\pa$ in $E$ (that is, $\pa$ is not a double point of $E$) we can find a function $f\in \mathfrak M_{X,0}$ with an isolated singularity at $0$ such that $(f)=(f)_{\Gamma} + f^{\ast}$ is a normal crossing divisor on $X_{\pi}$ and $(f)_{\Gamma} = D$ and such that its strict transform $f^*$ via $\pi$ intersects $E$ nontrivially at $\pa$.
	\end{enumerate}
\end{proposition}

%

We remark that the reason why the second part requires a stronger inequality is because we want the inequality of the first part to hold also after blowing up a free point of $E$.

\begin{proof}
	What is missing from \cite[Theorem~4.1]{CaubelNemethiPopescu-Pampu2006} with respect to the first part of our statement is the basepoint-freeness of $\mathcal{O}_{X_\pi}(-D)$.
	To see this, we recall the following two facts which are obtained in the proof \cite[Theorem~4.1]{CaubelNemethiPopescu-Pampu2006}: first, the natural application $H^0\big(\mathcal{O}_{X_{\pi}}(-D)\big) \to H^0\big(\mathcal{O}_E(-D)\big)$ is surjective \cite[Page 685, second to last paragraph]{CaubelNemethiPopescu-Pampu2006}; second, for every point $\pa \in E$, there exists a global section of $\mathcal{O}_E(-D)$ which is non-zero at $\pa$ \cite[Page 685, last paragraph]{CaubelNemethiPopescu-Pampu2006}. 
	Now fix a point $\pa \in E$ and consider the set
	\[
	H_{\pa}= \big\{ s \in H^{0}( E,\mathcal{O}_E(-D)) \,\big|\, s(\pa) =0\big\}
	\]
	which is the kernel of the linear map $H^{0}\big(E,\mathcal{O}_E(-D)\big) \to \mathbb{C}$ given by the evaluation of the sections at $\pa$. 
	By the second fact, it is a proper subspace of $H^{0}\big(E,\mathcal{O}_E(-D)\big)$ of codimension at least one. Finally, the function $f$ of the statement of the theorem is taken as a global section of $\mathcal{O}_{X_{\pi}}(-D)$ whose projection to $\mathcal{O}_E(-D)$ is generic \cite[Page 686, second paragraph]{CaubelNemethiPopescu-Pampu2006}. 
	We conclude that we can suppose the strict transforms of $f$ does not pass through $\pa$. This proves $(i)$.
	
	In order to prove part $(ii)$, let $D $ be an effective cycle satisfying inequality \eqref{eq:Ineq2} for every vertex $v\in V(\Gamma)$.
	Fix a free point $\pa \in E$ and consider the blowup $\sigma \colon ( X_{\pi'},E') \to (X_{\pi},E)$ with center $\pa$. 
	We denote by $\pi' = \pi \circ \sigma$, and by $E_w'$ the irreducible component of the exceptional divisor created by $\sigma$. 
	Consider the cycle $D' = \sigma^{\ast}(D) + E_w'$. 
	Note that $D' \cdot E'_w =-1  = - \val_{\Gamma_{\pi'}}(w) - 2g(w)$ and, if $v \neq w$ then $v$ is also a vertex of $\Gamma_{\pi}$, which implies that
	\[
	D' \cdot E_v' = D \cdot E_v + E'_w \cdot E_v' \leq D\cdot E_v + 1,
	\]
	and we conclude that $D'$ satisfies inequality \eqref{eq:Ineq1} for every vertex $v\in \Gamma_{\pi'}$. 
	It follows from part $(i)$ that there exists a function $f \in \mathfrak M_{X,0}$ with an isolated singularity at $0$ such that $(f)=(f)_{\Gamma_{\pi'}} + f^{\ast}_{\pi'}$ is a normal crossing divisor on $X_{\pi}'$, $(f)_{\Gamma_{\pi'}} = D'$ and $f^{\ast}_{\pi'} \cdot E_w = \val_{\Gamma_{\pi'}}(w) + 2g(w) =1$.
	We conclude that $(f)=(f)_{\Gamma_{\pi}} + f^{\ast}_{\pi}$ is such that $(f)_{\Gamma_{\pi}} = D$, and $(f)$ is a normal crossing divisor on $X_{\pi}$ outside a neighborhood of $\pa$. 
	At the point $\pa$, we know that the order of $f^{\ast}_{\pi}$ must be one, implying that it is smooth (since, after the blowup, $(f)_{\Gamma_{\pi'}}=D'= \sigma^{\ast}(D) + E_w'$). 
	Furthermore, $f^{\ast}_{\pi}$ must be transverse to the exceptional divisor, since its strict transform is transverse to a free point of $E_w$. 
	We therefore conclude that $(f)$ is a normal crossing divisor on $X_{\pi}$ and that $f^{\ast}_{\pi}$ intersects the free point $\pa$. 
	This proves $(ii)$.
\end{proof}

\section{Bound on multiplicities}
\label{sec:bound_multiplicity}

In this section we prove Proposition~\ref{prop:main} using the construction of Proposition~\ref{thm:CaubelNemethiPopescu-Pampu2006} and some basic commutative algebra.

\medskip

As mentioned in the introduction, thanks to \cite{Neumann1981} the datum of an homeomorphism class of a real 3-manifold $M$ that can be realized as the link of a normal surface singularity is equivalent to the datum of a weighted connected graph $\Gamma$ with negative-definite self-intersection matrix.Moreover, for a given $M$, there is an unique such graph $\Gamma$ that is minimal in the sense that it has no vertex of genus 0, valency at most two, and self-intersection -1. Let us therefore fix such a minimal graph $\Gamma$,
let $(X,0)$ be a normal surface singularity, and let $\pi \colon (X_{\pi},E) \to (X,0)$ be the minimal good resolution of $(X,0)$, and assume that the weighted dual graph of $\pi$ is $\Gamma$. 
Let $D$ be an integral effective divisor satisfying the inequality~\eqref{eq:Ineq1} of Proposition~\ref{thm:CaubelNemethiPopescu-Pampu2006} for every vertex $v$ of $\Gamma$, so that the line bundle $\mathcal{O}_{X_\pi}(-D)$ is basepoint-free. 
Since $(X,0)$ is normal, reasoning like at the beginning of the proof of Proposition~\ref{thm:CaubelNemethiPopescu-Pampu2006} this implies the existence of two functions $f,\, g \in \mathfrak M_{X,0}$, with an isolated singularity at $0$, whose total transforms by $\pi$ are given by
\[
(f) = (f)_{\Gamma} +f^{\ast} = D + f^{\ast}, \quad (g) = (g)_{\Gamma} +g^{\ast} = D + g^{\ast},
\]
where the strict transforms $f^{\ast}$ and $g^{\ast}$ are smooth disjoint curves.
Consider the map $\Psi= (f,g)\colon (X,0) \to (\mathbb{C}^2,0)$, which is a finite morphism. 
Then the degree $\deg(\Psi)$ of $\Psi$ can be computed as the number of points in a general fiber of $\Psi$, that is the number $\dim\big(\mathcal{O}_{X,0} / (f,g)\big)$ of intersection points $(f= \epsilon) \cap (g=\delta)$ for sufficiently small $\epsilon$ and $\delta$ in a neighborhood of $0$ in $ \mathbb{C}$.
Since $f^{\ast} \cdot g^{\ast}=0$, we conclude that this intersection multiplicity is equal to $D \cdot g^{\ast}$. 
By \cite[Theorem 2.6]{Laufer1971} we have $(g) \cdot D = \big((g)_{\Gamma} + g^{\ast}\big) \cdot  D = 0$, and so $\deg(\Psi)=  g^{\ast} \cdot  D  = -D^2$, which implies that $\dim\big(\mathcal{O}_{X,0} / (f,g)\big)= -D^2$. 
It now follows from \cite[Theorem~14.10]{Matsumura1980} that $\dim\big(\mathcal{O}_{X,0} / (f,g)\big) \geq m\big((f,g),\mathcal{O}_{X,0}\big)$, and from \cite[Formula~14.4]{Matsumura1980} that $m\big((f,g),\mathcal{O}_{X,0}\big) \geq m(\mathfrak{M}_{X,0},\mathcal{O}_{X,0}) = m(X,0)$. We deduce that $m(X,0)  \leq -D^2$, which concludes the proof of Proposition~\ref{prop:main}. \hfill\qed


\section{Bounding the number of $\mathcal L$-vectors}

The results of the previous sections are sufficient to prove the following weaker version of Theorem~\ref{thm:main}.

\begin{proposition}
\label{prop:weaker_version_main_theorem}
	Let $M$ be a real $3$-manifold. 
	There exists finitely many pairs $(\Gamma,L)$, where $\Gamma$ is a weighted graph and $L$ is a vector in $(\Z_{\geq0})^{V(\Gamma)}$, such that there exists a normal surface singularity $(X,0)$ satisfying the following conditions:
	\begin{enumerate}
		\item The link of $(X,0)$ is homeomorphic to $M$.
		\item $(\Gamma,L) = (\Gamma_\pi, L_\pi)$, where $\pi \colon X_{\pi} \to X$ is the minimal good resolution of $(X,0)$ which factors through the blowup of the maximal ideal of $(X,0)$.
	\end{enumerate}
\end{proposition}

\begin{proof}
As discussed at the beginning of Section~\ref{sec:bound_multiplicity}, the homeomorphism class of $M$ determines a minimal weighted graph $\Gamma_0$.
Let $(X,0)$ be a normal surface singularity whose minimal good resolution $\pi\colon X_\pi\to X$ has weighted dual graph $\Gamma_0$ and denote by $L=(l_v)_{v\in V(\Gamma_0)}$ the corresponding $\cal L$-vector and by $(m_v)_{v\in V(\Gamma_0)}$ the corresponding multiplicities. Denote now by $\pi'\colon X_{\pi'}\to X$ the minimal good resolution of $(X,0)$ that factors through the blowup of its maximal ideal.
We claim that there are finitely many possibilities for the weighted dual graph $\Gamma_{\pi'}$ of $\pi'$.
Indeed, the map $\pi'$ factors through $\pi$, and the resulting map $\alpha\colon X_{\pi'}\to X_\pi$ is a sequence of point blowups, each of which is centered at a basepoint of the family of generic hyperplane sections of $(X,0)$.
In particular, under each such blowup, the sum $\sum_v m_v l_v$ increases.
Since we have $\sum_{v\in V(\Gamma)} m_vl_v = - Z_{\max}^{\Gamma}(X,0)^2 \leq m(X,0)$ by \cite[Theorem 2.7]{Wagreich1970} (see also \cite[Theorem 2.18]{Nemethi1999}), and the latter is bounded by the integer $n_M$ from Proposition \ref{prop:main}, this implies that $\alpha$ consists of at most $n_M$ point blowups, which is sufficient to describe a finite list of graphs to which $\Gamma_{\pi'}$ belongs.
Moreover, since $m_v\geq1$ for all vertices $v$ of $\Gamma_{\pi'}$, we deduce that $l_v\leq n_M$ for all $v$, which proves that finitely many vectors in $\Z^{V(\Gamma_{\pi'})}$ can be realized as $\cal L$-vectors of a normal surface singularity.
\end{proof}

\begin{remark}\label{rem:bound_L}
	While we have striven to make the proof above as simple as possible, more optimal bounds on the number of realizable 
	vectors $L$	can be found via a more careful approach.
	For instance, not all vectors $L$ outputted by the proof above can be obtained as solution of a linear system of the form $L = Z_{\max}^{\Gamma}(X,0) \cdot E$, since $I_\Gamma^{-1}\cdot L$ needs not have integer coordinates in general.
	One can obtain a shorter list of possibly realizable $\cal L$-vectors by considering the smallest possible integral divisor $D$ to which Theorem~\ref{thm:CaubelNemethiPopescu-Pampu2006} applies (this divisor can be found very easily using the dual basis of the Lipman cone with respect to the intersection matrix of $\Gamma$),
	so that $Z_{\max}^{\Gamma}(X,0) \leq D$, which gives us a finite list of candidates for the maximal ideal divisor of any normal surface singularity realizing $\Gamma$, and therefore a much shorter list of possibilities for the vector $L$. We can then also reduce the list of possibilities for the graph $\Gamma_{\pi'}$ by only considering the possible maximal ideal divisors and $\cal L$-vectors above, which greatly reduces the number and combinatorics of the blowups in the morphism $\alpha$ appearing in our proof.
\end{remark}

\section{K\"ahler differentials and valuative invariants}
\label{sec:kahler}

In this section we introduce two valuative invariants associated with the sheaf of 2-forms on a normal surface singularity and prove some results that will allow us to use them to prove Theorem~\ref{thm:main} in Section~\ref{sec:proof_thm_A}.

\medskip

Given a normal surface singularity $(X,0)$, we consider the sheaf of \emph{K\"ahler 2-forms} $\Omega^2_X$ on $X$. 
We refer the reader to \cite[$\S$16]{Eisenbud1995} for the general definition; in this work it is enough to work with the following local description of its pullback to a resolution, proven in $\S$20.2 of \emph{loc.\ cit.}: given a local embedding $i \colon (X,0) \hookrightarrow (\mathbb{C}^N,0)$
and a good resolution $\pi\colon(X_{\pi},E) \to (X,0)$, then we have $\pi^{\ast}\Omega^2_X = (i \circ \pi)^{\ast} \Omega_{\mathbb{C}^N}^2$, where $\Omega_{\mathbb{C}^N}^2$ be the sheaf of differential $2$-forms on $\mathbb{C}^N$.

Now, since $X_{\pi}$ is smooth, the sheaf of 2-forms $\Omega_{X_{\pi}}^2$ is locally free of rank one. 
Consider the $\mathcal{O}_{X_{\pi}}$-subsheaf generated by the image of 
\(
\pi^{\ast} \colon \Omega_{X}^2 \to \Omega_{X_{\pi}}^2.
\)
This sheaf is of the form $\mathcal{F}_0(\pi) \,\Omega_{X_{\pi}}^2$, where $\mathcal{F}_0(\pi)$ is an ideal sheaf called the \emph{$0$-Fitting ideal} associated with $\pi$ (see \cite[$\S$~20.2]{Eisenbud1995} or \cite[$\S$2.3]{BelottodaSilvaBierstoneGrandjeanMilman2017}). 

We now recall the definition of a natural invariant associated with the sheaf $\Omega_X^2$, introduced in \cite[Definition~1.9]{FernexEinIshii2008} (see also \cite[Page~1259]{IshiiReguera2013} or \cite[\S~2.1]{FernexDocampo2014} for a point of view closer to the one we adopt here).
Given a vertex $v$ of $V(\Gamma_{\pi})$, the \emph{Mather discrepancy} $\hat k_v$ of $(X,0)$ along $v$ is defined as
\[
\hat{k}_v = \ord_v \big(\pi^{\ast}(\Omega_{X}^2)\big)= \ord_v\big(\mathcal{F}_0(\pi)\big).
\]

We also consider the \emph{residual ideal sheaf} $\mathcal{R}_0(\pi)$ of $\mathcal{F}_0(\pi)$ \cite[Definition~4.1]{BelottodaSilvaBierstoneGrandjeanMilman2017}, which is defined stalk-wise by setting
\begin{equation}\label{eq:Residual}
\mathcal{F}_0(\pi)_{\pa} = \mathcal{R}_0(\pi)_{\pa} \, \prod_{v \in V(\Gamma)\text{ s.t. } \pa \in E_v} {x_{v,\pa}}^{\hat{k}_v}
\end{equation}
for every closed point $\pa$ of $E$,
where $x_{v,\pa}$ is a reduced local equation for $E_v$ at $\pa$. 
Note that the order of vanishing  $\mbox{ord}_{\pa}(\mathcal{R}_0(\pi))$ of the ideal sheaf $\mathcal{R}_0(\pi)$ at a closed point $\pa$ of $E$ is zero except at at most finitely many points $\pa$; in particular its order along $E_v$ is zero for each $v \in V(\Gamma_{\pi})$.
Moreover, $\mathcal{R}_0(\pi)$ is trivial if and only if $\mathcal{F}_0(\pi)$ is principal, which is equivalent to the fact that $\pi$ factors through the Nash transform of $(X,0)$ (this result can be found in \cite[Theorem 2.5]{BelottodaSilvaBierstoneGrandjeanMilman2017}, but also seems to be implicit in other references, such as \cite{FernexEinIshii2008}).
In our context, it can also be seen from \cite[III, Theorem~1.2]{Spivakovsky1990} that $\pi$ factors through the Nash transform of $(X,0)$ if and only if the family of the polar curves of the generic plane projections has no basepoints on $X_\pi$.

The relation between the 0-Fitting ideal and generic polar curves can be seen explicitly as follows.
Let $(\Pi_{\cal D})_{\cal D \in \Omega}$ be the family of polar curves associated with the generic plane projections $\ell_{\cal D} \colon (X,0) \to (\C^2,0)$ of $(X,0)$ (with the notations of \cite[Section 2.2]{BelottodaSilvaFantiniPichon2019}, so that in particular it is an equisingular family in the sense of strong simultaneous resolution). 
Consider a closed point $\pa$ of $E$. 
For all $\cal D$ in $\Omega$, denote by $(\Pi_{\cal D, \pa}, 0)$ the union of the irreducible components of $(\Pi_{\cal D},0)$ whose strict transforms pass through $\pa$, and let $\Omega_{\pa}$ be the maximal Zariski open and dense subset of $\Omega$ such that the family $(\Pi_{\cal D, \pa})_{\cal D \in \Omega_{\pa}}$ is equisingular. 
Note that the latter condition is equivalent to ask that the number of irreducible components of $\Pi_{\cal D, \pa}$ is constant for every $\cal D$ in $\Omega_{\pa}$ and that, by definition, $\Pi_{\cal D, \pa}$ is nonempty if and only if $\pa$ is a basepoint of the family of polar curves $(\Pi_{\cal D})_{\cal D \in \Omega}$.
Then the ideal $\mathcal{R}_0(\pi)_{\pa}$ defines exactly the family consisting of the strict transforms via $\pi$ of the curves $\Pi_{\cal D, \pa}$, for $\cal D$ in $\Omega_\pa$.
Indeed, suppose that $\pa$ is a free point of $E$ (the case when $\pa$ is a double point is completely analogous) belonging to the component $E_v$ and that $(x,y)$ are local coordinates for $X_\pi$ at $\pa$ such that $E_v$ is locally defined by the equation $x=0$.
If locally at $\pa$ we write $\phi_{\cal D}(x,y)=(\ell_{\cal D} \circ \pi)(x,y)=\big(z_1(x,y),z_2(x,y)\big)$, then $\cal F_0(\pi)_\pa$ is generated by the forms 
\(
\phi_{\cal D}^* (dz_1 \wedge dz_2)_\pa,
\)
where $\cal D$ varies in $\Omega_\pa$, $\ell_{\cal D}\colon (X,0) \to (\C^2,0)$ is the associated plane projection, and $dz_1 \wedge dz_2$ is a standard volume form on $(\C^2,0)$.
Observe that we can write
\[
\phi_{\cal D}^* (dz_1 \wedge dz_2)_\pa = \mbox{Jac}\big(\phi(x,y)\big)(dx \wedge dy) = x^{\hat k_v}f_{\cal D}(x,y)(dx \wedge dy) 
\]
for some $f_{\cal D}$ in $\cal O_{X_\pi,\pa}$.
Since $\pi$ is an isomorphism outside of $E$, it follows that $f_{\cal D}(x,y)=0$ is the local equation at $\pa$ of the strict transform via $\pi$ of the critical locus of $\ell_{\cal D}$, which is by definition the polar curve of $\ell_{\cal D}$.
In other words, $f_{\cal D}(x,y)=0$ is the local equation of the curve $(\Pi_{\cal D, \pa}, 0)$ defined above.
We refer to \cite[Chapter~3, Section~1]{Spivakovsky1990} for further details.

Although the Mather discrepancies $\hat k_v$ depend on the analytic structure of $(X,0)$, we will show that the weighted dual graph $\Gamma_{\pi}$ of a good resolution of singularities $\pi\colon(X_{\pi},E) \to (X,0)$ is enough to bound their growth under blowups. We start the following result.

\begin{theorem}
	\label{thm:sheaf_2-forms}
	Let $\Gamma$ be a weighted graph.
	Then there exists a divisor $D$ on $\Gamma$ such that, for every normal surface singularity $(X,0)$ and every good resolution $\pi \colon (X_{\pi},E) \to (X,0)$ of $(X,0)$ whose weighted dual graph is $\Gamma$, the line bundle $\mathcal{O}_{X_{\pi}}(-D)$ is basepoint-free and there exists a subsheaf $\Omega \subset \Omega^2_{X}$ of the sheaf of K\"ahler differentials of $X$, such that the pullback $\pi^{\ast}(\Omega)$ generates a $\mathcal{O}_{X_{\pi}}$-subsheaf of $\Omega^2_{X_{\pi}}$ of the form $\mathcal{J}(\pi) \Omega^{2}_{X_{\pi}}$, where $\mathcal{J}(\pi)$ is a principal ideal sheaf equivalent to $\mathcal{O}_{X_{\pi}}(-D)$.
\end{theorem}

Observe that if we had only been interested in the existence of a basepoint-free subsheaf of $\Omega_{X_{\pi}}^2$ we could have obtained it from \cite[Theorem 3.1]{Laufer1983}.
However, for our applications, and namely to prove Lemma~\ref{lem:compatibility} below, it is important to require that this subsheaf is the pullback of a subsheaf of $\Omega_{X}^{2}$.

\begin{proof}
	Let $D_1 = \sum_{v \in V(\Gamma)} \alpha_v E_v$ and $D_2= \sum_{v \in V(\Gamma)} \beta_v E_v$ be two integral effective divisors satisfying inequality \eqref{eq:Ineq2} for every vertex $v \in V(\Gamma)$, and which verify the following open condition in the Lipman cone:
	\(
	\alpha_v \beta_w - \beta_v \alpha_w \neq 0,
	\) for all vertexes $v$ and $w$ of $V(\Gamma)$ which are connected by an edge. We set $D = D_1 + D_2 - E$.
	
	Note that by Proposition \ref{thm:CaubelNemethiPopescu-Pampu2006}(i), both line bundles $\mathcal{O}(-D_1)$ and $\mathcal{O}(-D_2)$ have no basepoints. 
	This fact together with the normality of $(X,0)$ implies the existence of two ideals $\mathcal{I}_1$ and $\mathcal{I}_2$ of $\mathfrak M_{X,0}$ whose pullbacks by $\pi$ are equivalent to  $\mathcal{O}(-D_1)$ and $\mathcal{O}(-D_2)$ respectively. 
	By considering a local embedding $i \colon (X,0)\hookrightarrow (\mathbb{C}^N,0)$, we may consider without loss of generality that $\mathcal{I}_1$ and $\mathcal{I}_2$ are ideals of $\mathfrak M_{(\mathbb{C}^N,0)}$. 
	Consider the $\mathcal{O}_{(\mathcal{C}^N,0)}$ submodule of $2$-forms $\Omega \subset \Omega_{\mathbb{C}^N}^2$ generated by all $2$-forms of the form $df_1 \wedge df_2$ with $f_1$ in $\mathcal{I}_1 $ and $f_2$ in $\mathcal{I}_2 $. We claim that $i^{\ast}\Omega$ is the desired sheaf. 
	
	Indeed, since $X_{\pi}$ is a smooth surface, the sheaf $\Omega^2_{X_{\pi}}$ is everywhere an $\mathcal{O}_{X_{\pi}}$ free module of rank $1$, and the pullback $(\pi \circ i)^{\ast}\Omega $ generates a sheaf of the form $\mathcal{J} (\pi) \Omega_{X_{\pi}}$, where $\mathcal{J}(\pi)$ is an ideal sheaf. 
	We now verify that $\mathcal{J}(\pi)$ is a principal ideal equivalent to $\mathcal{O}(-D)$; note that it is enough to verify this statement at each point $\pa \in E$. 
	We divide the proof in two parts depending on the nature of $\pa$ as follows.
	
	Suppose first that $\pa$ is a free point of $E_v$.
	There exists a local coordinate system $(x,y)$ centered at $\pa$ such that $E_v$ is locally equal to $(x=0)$. By Proposition \ref{thm:CaubelNemethiPopescu-Pampu2006}, there are two functions $f_1$ in $\mathcal{I}_1$ and $f_2$ in $\mathcal{I}_2$ such that, apart from a change of coordinates, $(f_1 \circ \pi) (x,y) = x^{\alpha_v} y$ and $(f_2\circ \pi)(x,y) = x^{\beta_v}$. It follows that
	\[
	\pi^{\ast} (df_1 \wedge df_2) = \beta_v x^{\alpha_v + \beta_v -1} dx \wedge dy.
	\]
	which implies that $\mathcal{J}(\pi)_{\pa} \supset (x^{\alpha_v + \beta_v -1})$; the other inclusion easily follows from the construction, so $\mathcal{J}(\pi)_{\pa} = (x^{\alpha_v + \beta_v -1})$.
	
	Suppose now that $\pa \in E_v \cap E_w$ is a double point of $E$. 
	There exists a local coordinate system $(x,y)$ centered at $\pa$ such that $E_v= (x=0)$ and $E_w=(y=0)$. By Proposition \ref{thm:CaubelNemethiPopescu-Pampu2006}, there are two functions $f_1 \in \mathcal{I}_1$ and $f_2 \in \mathcal{I}_2$ such that $(f_1 \circ \pi) (x,y) = x^{\alpha_v} y^{\alpha_w}U_1(x,y)$ and $(f_2\circ \pi)(x,y) = x^{\beta_v}y^{\beta_w}U_2(x,y)$ where $U_1(0) \neq 0$ and $U_2(0)\neq 0$. 
	Since $\alpha_v \beta_w - \alpha_w\beta_v \neq 0$, up to a change of coordinates we may suppose that $U_1 \equiv  U_2 \equiv 1$, so that
	\[
	\pi^{\ast} (df_1 \wedge df_2) = (\alpha_v \beta_w - \alpha_w\beta_v ) \, x^{\alpha_v + \beta_v -1} y^{\alpha_w + \beta_w -1} dx \wedge dy.
	\]
	which implies that $\mathcal{J}(\pi)_{\pa} \supset (x^{\alpha_v + \beta_v -1} y^{\alpha_v + \beta_v -1})$; the other inclusion follows easily from the construction, so that $\mathcal{J}(\pi)_{\pa} = (x^{\alpha_v + \beta_v -1} y^{\alpha_v + \beta_v -1})$.
	This concludes the proof.
\end{proof}

Fix a graph $\Gamma$, let us fix once and for all a divisor $D$ on $\Gamma$ given to us by Theorem~\ref{thm:sheaf_2-forms}, and write
\[
D = \sum_{v\in V(\Gamma)} \nu_vE_v.
\]
Note that, if $(X,0)$ is a normal surface singularity, $\pi \colon (X_{\pi},E) \to (X,0)$ is a good resolution of $(X,0)$ whose weighted dual graph is $\Gamma$, and $\Omega$ is the subsheaf of $\Omega^{2}_X$ given by Theorem~\ref{thm:sheaf_2-forms}, (so that the pullback $\pi^{\ast}(\Omega)$ generates a sheaf of the form $\mathcal{J}(\pi) \, \Omega_{X_{\pi}}^2$, where $\mathcal{J}(\pi)$ is a principal ideal equivalent to the basepoint-free line bundle $\mathcal{O}(-D)$), we also have 
\(
\nu_v = \ord_v(\pi^{\ast}(\Omega)) = \ord_v(\mathcal{J}(\pi)) .
\)
This leads to a natural extension of the definition of these invariants to a vertex $v$ of the dual graph $\Gamma_\pi$ of any good resolution $\pi'$ of $(X,0)$ that factors through $\pi$, by setting
\(
\nu_v = \ord_v(\pi'^{\ast}(\Omega)) = \ord_v(\mathcal{J}(\pi')) . 
\)
Note that this definition only depends on the combinatorics of the sequence of blowups required to pass from $\pi$ to $\pi$'.
This is made clear by the following lemma, and also explains how the Mather discrepancies behave after blowups.

\begin{lemma}\label{lem:bound}
	Let $\Gamma$ and $D$ be as above.
	Let $(X,0)$ be a normal surface singularity, let $\pi$ be a good resolution of $(X,0)$ whose weighted dual graph is $\Gamma$, let $\pi'$ be another good resolution that factors through $\pi$, and let $\pi'' = \pi' \circ \sigma$ be the composition of $\pi'$ with a blowup $\sigma \colon (X_{\pi''},E'') \to (X_{\pi'},E')$ centered at a closed point $\pa$ of $E'$.
	Write $\sigma^{-1}(\pa)=E_w$.
	Then we have:
	\begin{enumerate}
		\item if $\pa $ is a smooth point of $E'$ contained in the component $E_v$ we have
		\[
		\nu_w =  \nu_v+1 \quad\quad \text{and} \quad\quad \hat k_w = \hat k_v + 1 + \ord_{\pa}\big(\mathcal{R}_0(\pi')\big); 
		\]
		\item if $\pa$ is a double point of $E'$ lying at the intersection of two components $E_{v}$ and $E_{v'}$ we have
		\[
		\nu_w = \nu_{v} + \nu_{v'} +1 \quad\quad \text{and} \quad\quad \hat k_w = \hat k_v + \hat k_{v'} +1 + 
		\ord_{\pa}\big(\mathcal{R}_0(\pi')\big).
		\]
	\end{enumerate}
	
\end{lemma}

\begin{proof}
We divide the proof in two cases depending on the nature of $\pa \in E'$ and perform two computations similar to the ones of \cite[$\S$2.4]{BelottodaSilvaBierstoneGrandjeanMilman2017}.	
Suppose first that $\pa$ is a free point of $E'$. 
Then there exists a local coordinate system $(x,y)$ centered at $\pa$ such that $E_v$ is defined locally by $(x=0)$. 
Consider a 2-form germ $\omega$ at $\pa$, and write \(\omega = x^{\alpha} f(x,y) dx \wedge dy\). 
Note that the order of the pullback form $\sigma^{\ast}(\omega)$ over $E_w$ may be computed at any general point of $E_w$. 
We consider the origin of the $x$-chart $(x,y) = (\tilde{x},\tilde{x}\tilde{y})$, where we obtain
\[
\begin{aligned}
\sigma^{\ast}(\omega) &= \tilde{x}^{\alpha} f(\tilde{x},\tilde{x}\tilde{y}) \, d\tilde{x}  \wedge d(\tilde{x}\tilde{y})\\
& = \tilde{x}^{\alpha+1 + \ord_{\pa}(f)} \tilde{f}(\tilde{x},\tilde{y})\, d\tilde{x} \wedge d\tilde{y},
\end{aligned}
\]
where $\tilde{f}(\tilde{x},\tilde{y})$ is such that $\tilde{f}(0,\tilde{y}) \not\equiv 0$ and $E_w = (\tilde{x}=0)$. It follows that 
\[
\ord_w\big(\sigma^{\ast}(\omega)\big) = \alpha + 1 + \ord_{\pa}(f).
\]
Now, note that $\mathcal{J}(\pi')$ is a principal ideal, so that $\mathcal{J}(\pi')_{\pa} = (x^{\nu_v})$. 
In other words, the differential form \(\omega = x^{\nu_v} dx \wedge dy\) belongs to the subsheaf generated by the pullback $\pi^{\ast}(\Omega)$, and we easily conclude that \(\nu_w = \nu_v+1.\) 
Finally, note that $\mathcal{F}_0(\pi')_{\pa} = x^{\hat{k}_v}(f_1,\ldots,f_k)$, where $f_1,\ldots,f_k$ are generators of $\mathcal{R}_{0}(\pi')_{\pa}$. 
In particular, the differential forms \(\omega_i = x^{\hat{k}_v} f_i(x,y) dx \wedge dy\) belongs to the subsheaf generated by the pullback $\pi^{\ast}(\Omega_{X}^2)$, and we easily conclude that \(\hat k_w = \hat{k}_v + 1 + \ord_{\pa}(\mathcal{R}_0(\pi')).\) 

Suppose now that $\pa \in E_v \cap E_{v'}$ is a double point of $E'$.
Then there exists a local coordinate system $(x,y)$ centered at $\pa$ such that $E_v =(x=0)$ and $E_{v'} = (y=0)$. 
Consider a 2-form germ $\omega$ at $\pa$, and write \(\omega = x^{\alpha} y^{\beta} f(x,y) dx \wedge dy\). 
Note that the order of the pullback form $\sigma^{\ast}(\omega)$ over $E_w$ may be computed at any generic point of $E_w$. 
We consider the origin of the $x$-chart $(x,y) = (\tilde{x},\tilde{x}\tilde{y})$, where we get
\[
\begin{aligned}
\sigma^{\ast}(\omega) &= \tilde{x}^{\alpha} (\tilde{x}\tilde{y})^{\beta} f(\tilde{x},\tilde{x}\tilde{y}) \, d\tilde{x}  \wedge d(\tilde{x}\tilde{y})\\
& = \tilde{x}^{\alpha+\beta+1 + \ord_{\pa}(f)} \tilde{y}^{\beta}\tilde{f}(\tilde{x},\tilde{y})\, d\tilde{x} \wedge d\tilde{y}
\end{aligned}
\]
where $\tilde{f}(\tilde{x},\tilde{y})$ is such that $\tilde{f}(0,\tilde{y}) \not\equiv 0$ and $E_w = (\tilde{x}=0)$. 
It follows that 
\[
\ord_w\big(\sigma^{\ast}(\omega)\big) = \alpha + \beta + 1 + \ord_{\pa}(f).
\]
Now, note that $\mathcal{J}(\pi')$ is a principal ideal, so that $\mathcal{J}(\pi')_{\pa} = (x^{\nu_v} y^{\nu_{v'}})$. 
In other words, the differential form \(\omega = x^{\nu_v} y^{\nu_{v'}} dx \wedge dy\) belongs to the subsheaf generated by the pullback $\pi^{\ast}(\Omega)$, and we easily conclude that \(\nu_w = \nu_v+ \nu_{v'}+1.\)
Finally, note that $\mathcal{F}_0(\pi')_{\pa} = x^{\hat{k}_v}y^{\hat{k}_{v'}}(f_1,\ldots,f_k)$, where $f_1,\ldots,f_k$ are generators of $\mathcal{R}_{0}(\pi')_{\pa}$. 
In particular, the differential forms \(\omega_i = x^{\hat{k}_v} y^{\hat{k}_{v'}} f_i(x,y) dx \wedge dy\) belong to the subsheaf generated by the pullback $\pi^{\ast}(\Omega_{X}^2)$, and thus \(\hat k_w = \hat k_v + 1 + \ord_{\pa}(\mathcal{R}_0(\pi')).\)
\end{proof}

Observe that, while the divisor $D$ that we have chosen, and therefore the integers $\nu_v$, only depend on the graph $\Gamma$ and not on the choice of $(X,0)$ and $\pi$, the sheaf $\mathcal F_0(\pi)$, and therefore the integers $\hat k_v$, depends on the analytic structure of $(X,0)$.
However, as a consequence of Theorem~\ref{thm:sheaf_2-forms} we deduce that $\hat k_v$ is bounded from above by the topological invariant $\nu_v$:

\begin{lemma}\label{lem:compatibility}
	Let $\pi' \colon (X_{\pi'},E') \to (X,0)$ be a good resolution of $(X,0)$ which factors through $\pi \colon (X_{\pi},E) \to (X,0)$. Then for every vertex $v$ of $\Gamma_{\pi'}$ we have
	\[
	\hat{k}_v \leq \nu_v.
	\]
	Moreover, if this inequality is an equality for some vertex $v$ of $\Gamma_{\pi'}$, then the family of the polar curves of the generic plane projections of $(X,0)$ has no basepoint at a free point of $E_{v}$, and if all inequalities for all vertices $v$ are equalities then the family of the polar curves of the generic plane projections of $(X,0)$ has no basepoint at all.
\end{lemma}

\begin{proof}
	Let us denote by $\mathcal{J}(\pi') \, \Omega_{X_{\pi'}}$ and $\mathcal{F}_0(\pi')\, \Omega_{X_{\pi'}}$ the subsheaves generated by the pullback of $\Omega$ and $\Omega_{X}^2$ respectively. Since $\Omega \subset \Omega_X^2$, we conclude that $\mathcal{J}(\pi') \subset \mathcal{F}_0(\pi')$, which implies the desired inequality.
	Now, suppose that the inequality is an equality over a vertex $v\in \Gamma_{\pi'}$.
	Since $\pi'$ factors through $\pi$, and $\mathcal{J}(\pi)$ is a principal ideal, we conclude that $\mathcal{J}(\pi')$ is also a principal ideal (indeed, $\mathcal{J}(\pi')$ is the pullback of $\mathcal{J}(\pi)$ multiplied by the Jacobian determinant of a sequence of point blowups, see Lemma~\ref{lem:bound}).
	It therefore follows that for all free point $\pa \in E_v$, the localizations $\mathcal{J}(\pi')_{\pa} = \mathcal{F}_0(\pi')_{\pa}$ coincide. Moreover, if the inequality is an equality over two vertices $v, \, v'\in \Gamma_{\pi'}$ that are connected by an edge, then the localizations $\mathcal{J}(\pi')_{\pa} = \mathcal{F}_0(\pi')_{\pa}$ coincide over any point $\pa \in E_v \cap E_{v'}$.
	We conclude by the results \cite[III, Theorem 1.2]{Spivakovsky1990} and \cite[Theorem 2.5]{BelottodaSilvaBierstoneGrandjeanMilman2017} already discussed above.
\end{proof}


\section{Proof of Theorem~\ref{thm:main}}
\label{sec:proof_thm_A}

The goal of this section is to complete the proof of our main result, Theorem~\ref{thm:main}.

\medskip

As we can rely on Proposition~\ref{prop:weaker_version_main_theorem}, we now turn our attention to the polar invariants, and we can assume without loss of generality that $\pi\colon X_\pi \to X$ is a good resolution of the normal surface singularity $(X,0)$ which factors through the blowup of its maximal ideal, and that the $\cal{L}$-vector $L=\{l_v\}_{v\in V(\Gamma_\pi)}$ is given. By the L\^e--Greuel--Teissier formula \cite[Theorem~5.1.1]{LeTeissier1981} (see also \cite[Proposition~5.1]{BelottodaSilvaFantiniPichon2019}), the multiplicity $m(\Pi,0)$ of the polar curve $\Pi$ of a generic projection $\ell\colon (X, 0) \to (\C^2, 0)$, is equal to $m(X,0)-\chi(F_t)$, where $F_t$ denotes the Milnor--L\^e fiber of a generic linear form on $(X,0)$. 
Now, for each vertex $v$ of $V(\Gamma_{\pi})$, let $N(E_v)$ be a small tubular neighborhood of the corresponding component $E_v$ of $\pi^{-1}(0)$ obtained as the total space of a normal disc bundle on $E_v$, and set
\[
\cal{N}(E_v) = \overline{N(E_v) \setminus  \bigcup_{E_w\neq E_v} N(E_w)}.
\]
Then by additivity of the Euler characteristic, we can compute $\chi(F_t)$ as
\begin{align*}
\chi(F_t) & = \sum_{v\in V(\Gamma_{\pi})}\chi\big(\cal N(E_v) \cap F_t\big) =  \sum_{v\in V(\Gamma_{\pi})}m_v\big(\chi\big(\cal N(E_v)\cap E_v\big)-l_v\big) \\
& =\sum_{v\in V(\Gamma_{\pi})}m_v\big(2-2g(E_v)-\val_{\Gamma_{\pi}}(v)-l_v\big),
\end{align*}
where the second equality makes use of the fact that $\pi$ factors through the blowup of the maximal ideal of $(X,0)$, and therefore the strict transform via $\pi$ of a generic hyperplane section of $(X,0)$ consists of $l_v$ curvettes on each component $E_v$.	
This shows that $\chi(F_t)$ only depends on the weighted graph $\Gamma_{\pi}$ and on the $\cal L$-vector $L$, so that we obtain an explicit bound for $m(\Pi,0)$ as well.

Observe that what we have proven so far suffices to deduce Corollary~\ref{cor:main}, since $m(\Pi,0)$ equal to the sum $\sum_{v\in V(\Gamma_{\pi})}m_vp_v$ and so the value $p_v$ for any vertex $v$ of $\Gamma_{\pi}$ is bounded as well.
What is sensibly harder to show, and requires the invariants introduced in Section~\ref{sec:kahler}, is the fact that the topology of $(X,0)$ determines a finite family of dual graphs for the minimal resolution factoring through the Nash blowup of $(X.0)$.
In order to do this, we introduce an auxiliary numerical invariant based on the following local version of the invariants $p_v$.
Following the notation of the previous section, given a closed point $\pa$ of $E$ pick $\cal D$ in $\Omega_\pa$, consider the associated curve germ $(\Pi_{\cal D, \pa}, 0)$, and set $p_v(\pa) = \Pi^*_{\cal D, \pa} \cdot E_v$, where $^*$ denotes the strict transform under $\pi$.
Then, since $\Pi_{\cal D, \pa}$ is nonempty if and only if $\pa$ is a basepoint of the family of polar curves $(\Pi_{\cal D})_{\cal D \in \Omega}$, we deduce that $p_v(\pa) \neq 0$ if and only if $\pa$ is a basepoint which belongs to $E_v$.

\begin{definition}
Given a point $\pa \in E$, consider the quadruple of integers
\[
\mathrm{Aux}(\pa)=\big(m(\Pi_{\cal D, \pa}), \epsilon_1(\pa),\epsilon_2(\pa),\beta(\pa)\big)
\]
whose last three entries are defined as
\begin{align*}
\epsilon_1(\pa) = \sum_{v}p_v(\pa),
 \quad \quad
\epsilon_2(\pa) = \max_v\big\{p_v(\pa),p_w(\pa)\big\},
 \quad \quad
\beta(\pa)= \max_v\big\{\nu_{v} - \hat k_v\big\},
\end{align*}
where the sum and the two maxima run over the set of vertices $v$ of $\Gamma_\pi$ such that $\pa$ belongs to $E_v$.
\end{definition}

In the rest of the section we consider the invariants $\Aux(\pa)$ as elements of $\Z^4$ equipped with the lexicographic order.
Note that all the entries of $\Aux(\pa)$ are positive integers: this is immediate for the first three, and a direct consequence of Lemma~\ref{lem:compatibility} for $\beta(\pa)$.
Moreover, observe that we have
\begin{equation}\label{eq:top_bound_for_Aux}
\Aux(\pa) \leq \big(m(\Pi,0), m(\Pi,0), m(\Pi,0), \max \{\nu_v\,|\,v\in V(\Gamma)\} \big),
\end{equation}
that is the invariant $\Aux(\pa)$ is bounded from above by a quadruple
which does not depend on the choice of $(X,0)$ and $\pi$ but only on the topology of $(X,0)$.

We are now ready to see that, as a consequence of the computations of Lemma~\ref{lem:bound}, the auxiliary invariant always drops after a blowup.

\begin{lemma}
\label{lem:induction}
Let $\pi\colon (X_{\pi},E) \to (X,0)$ be a good resolution of $(X,0)$ that factors through the blowup of its maximal ideal, let $\pa$ be a closed point of $E$ such that $m(\Pi_{\cal D, \pa})  \neq 0$, and let $\pi' = \pi \circ \sigma$ be the composition of $\pi$ with the blowup $\sigma \colon (X_{\pi'},E') \to (X_{\pi},E)$ with center $\pa$.
Then, for every closed point $\pb$ of $\sigma^{-1}(\pa)$ we have
\[
\mathrm{Aux}(\pb) < \mathrm{Aux}(\pa).
\]
\end{lemma}
\begin{proof}
Set $E_w=\sigma^{-1}(\pa)$.
Since $\pi$ factors through the blowup of the maximal ideal of $(X,0)$, the multiplicity $m(\Pi_{\cal D, \pa},0)$ of $\Pi_{\cal D, \pa}$ at $0$ can be computed as the sum
\(
m(\Pi_{\cal D, \pa},0) = \sum_{v \in V(\Gamma)} m_v p_v(\pa).
\)
If $\pb_1, \ldots,\pb_r$ are the basepoints of   $(\Pi_{\cal D})_{\cal D \in \Omega}$ on $E_w$,  then  for all $i = 1,\ldots, r,$ we have $\Omega_{\pb_i} \subset \Omega_{\pa}$ and thus $(\Pi_{\cal D, \pb_i}, 0) \subset (\Pi_{\cal D, \pa}, 0)$.
Moreover, for all $\cal D \in \bigcap_{i=1}^r \Omega_{\pb_i}$ and all $i, j$ with $i \neq j$, the curve germs $(\Pi_{\cal D, \pb_i}, 0)$ and $(\Pi_{\cal D, \pb_j}, 0)$ have no irreducible components in common, which implies that
\(
m(\Pi_{\cal D, \pa},0) \geq  \sum_{i=1}^r m(\Pi_{\cal D, \pb_i},0).
\)
It follows that if $r>1$, then $ m(\Pi_{\cal D, \pb_i}) <  m(\Pi_{\cal D, \pa})$ for every $i=1,\ldots,r$. We may therefore suppose that there is an unique closed point $\pb$ of $E_w$ which is a basepoint for the family of polar curves, that is that $r=1$. 
Moreover, if $ m(\Pi_{\cal D, \pb}) <  m(\Pi_{\cal D, \pa})$ then there is nothing to prove, and so we may further suppose that $m(\Pi_{\cal D, \pb}) =  m(\Pi_{\cal D, \pa})$. We now divide the proof in four parts, depending on the nature of $\pa$ and $\pb$.

First, suppose that $\pa$ is a double point belonging to $E_{v_1} \cap E_{v_2}$ and that $\pb$ is a free point of $E_w$. 
Then we have
\[
m_{v_1}p_{v_1}(\pa) + m_{v_2}p_{v_2}(\pa) = m_{w} p_{w}(\pb) = (m_{v_1} + m_{v_2})p_{w}(\pb).
\]
It easily follows that $\epsilon_1(\pb) =  p_{w}(\pb) < p_{v_1}(\pa)+p_{v_2}(\pa) = \epsilon_1(\pa)$.

Second, suppose that $\pa$ is a double point belonging to $E_{v_1} \cap E_{v_2}$ and that $\pb$ is a double point, without loss of generality say that $\pb$ belongs to $E_{v_1} \cap E_w$, and note that
\[
m_{v_1}p_{v_1}(\pa) + m_{v_2}p_{v_2}(\pa) = m_{v_1}p_{v_1}(\pb) + (m_{v_1}+m_{v_2}) p_{w}(\pb),
\]
so that we have $\epsilon_1(\pb) <\epsilon_1(\pa)$.

Third, suppose that $\pa$ is a free point of $E_{v}$ and that $\pb$ is a double point, say that $\pb$ belongs to $E_v \cap E_w$. Then we have
\[
m_v p_{v}(\pa) = m_v p_{v}(\pb) + m_w p_{w}(\pb) = m_v \big(p_{v}(\pb) + p_{w}(\pb)\big),
\]
and therefore $\epsilon_1(\pa) = \epsilon_1(\pb)$ and $\epsilon_2(\pb) < \epsilon_2(\pa)$.

Finally, suppose that $\pa$ is a free point of $E_{v}$ and that $\pb$ is a free point of $E_w$, so that we have
\[
m_vp_{v}(\pa) = m_w p_{w}(\pb) = m_v p_{w}(\pb),
\]
which implies that $\epsilon(\pa)=\epsilon(\pb)$. Now, recall that $\beta(\pa) >0$ by the hypothesis that $ m(\Pi_{\cal D, \pa})>0$ and by Lemma \ref{lem:compatibility}. Since $\pa$ is a basepoint for the family of generic polar curves, we conclude that it is contained in the zero locus of the residual ideal sheaf $\mathcal{R}_0(\pi)$, see \eqref{eq:Residual}, so that $\ord_{\pa}(\mathcal{R}_0(\pi)) >0$. It now follows from Lemma~\ref{lem:bound} that
\(
\beta(\pb) < \beta(\pa),
\)
finishing the proof.
\end{proof}

We have now collected all the ingredients we need to show that the number of blowups needed from any given resolution of $(X,0)$ to achieve factorization through its Nash transform admits an upper bound that only depends on the topology of $(X,0)$.
To see this, let $\Gamma$ be any weighted graph which can be realized as the dual graph of some good resolution of $(X,0)$ factoring through its maximal ideal.
Then the auxiliary invariants of the closed points of the exceptional divisor of such resolution are always bounded by the 4-tuple
\[
\mathrm{Aux}(\Gamma) = \max_{X',\pi',\pa} \{\Aux(\pa)\}
\]
where the maximum is taken over the set of triples $\big((X',0),\pi',\pa\big)$, where $(X',0)$ is a normal surface singularity, $\pi' \colon X'_{\pi'}\to X'$ is a good resolution of $(X',0)$ realizing the weighted graph $\Gamma$, and $\pa$ is a closed point of the exceptional divisor $(\pi')^{-1}(0)$ of $\pi'$.
Observe that, as noted in \eqref{eq:top_bound_for_Aux} the invariant $\mathrm{Aux}(\Gamma)$ is bounded from above by an element of $\Z^4$ which only depends on the topology of $(X,0)$.
It follows immediately from Lemma~\ref{lem:induction} that if $\pi \colon X_\pi\to X$ is a good resolution of $(X,0)$ which does not factor through its Nash transform and $\pi'\colon X_\pi\to X$ is the good resolution of $(X,0)$ obtained by blowing up once every basepoint of the family of generic polar curves of $(X,0)$, then we have
\[
\max_{\pb\in(\pi')^{-1}(0)}\{\mathrm{Aux}(\pb)\}
	<
\max_{\pa\in\pi^{-1}(0)}\{\mathrm{Aux}(\pa)\} \leq \mathrm{Aux}(\Gamma).
\]
The result we are after follows now immediately by induction, since the weighted dual graph $\Gamma_{\pi'}$ belongs to a finite family of dual graphs which only depends on the weighted dual graph $\Gamma_\pi$ and on $\mathrm{Aux}(\Gamma)$.
Indeed, the family of generic polar curves of $(X,0)$ can have at most $m(\Pi,0)$ basepoints, and we have already proven that $m(\Pi,0)$ is bounded by the topology of $(X,0)$, so that the number of blowups, and hence the combinatorics of $\Gamma_{\pi'}$, is itself bounded by the topology of $(X,0)$.
This concludes the proof of Theorem~\ref{thm:main}.\hfill\qed


\section{Polar exploration}

We have discussed in Remark~\ref{rem:bound_L} how to give sharper bounds on the number of realizable $\cal L$-vectors on a given resolution graph $\Gamma$.
On the other hand, the bound on the number of $\cal P$-vectors in Theorem~\ref{thm:main}, being solely based on the polar multiplicities given by the L\^e--Greuel--Teissier formula, gives no information on the position of polar curves relative to the hyperplane sections. 
We now discuss restrictions on these relative positions, thus providing a sharper bound to the number of realizable $\cal P$-vectors and a better understanding on the polar geometry of singularities realizing $\Gamma$. 
For this, we can shift our focus to the following situation, which we refer to as the problem of \emph{polar explorations}: given a resolution graph $\Gamma$ and a $\cal L$-vector $L$ on $\Gamma$, how can we give geometric conditions on a $\cal P$-vector $P$ such that the triplet $(\Gamma,L,P)$ may be realizable?

In this section we describe two distinct tools that can be very effective in addressing this question.

\subsection{Inner rates and the Laplacian formula}
Assume that $(X,0)$ is a normal surface germ realizing the pair $(\Gamma,L)$. 
The first tool we make use of is a result on the structure of this germ with respect to its \emph{inner metric} $d_{\inn}$, which is defined up to a bi-Lipschitz homeomorphism by embedding $(X,0)$ in a smooth germ $(\C^n,0)$ and considering the arc-length on $X$ induced by the usual Hermitian metric of $\C^n$.
Denote by $S_{\epsilon}$ the sphere in $\C^n$ having center $0$ and radius $\epsilon>0$.
Given two distinct curve germs $(\gamma,0)$ and $(\gamma',0)$ on $(X,0)\subset(\C^n,0)$, the \emph{inner contact} between $\gamma$ and $\gamma'$ is the rational number $q_{\inn}=q_{\inn}(\gamma, \gamma')$ defined by 
\[
d_{\inn} \big(\gamma \cap S_{\epsilon}, \gamma' \cap S_{\epsilon}\big) = \Theta(\epsilon^{q_{\inn}}),
\]
where given two function germs $f,g\colon \big([0,\infty),0\big)\to \big([0,\infty),0\big)$ we write $f(t) = \Theta \big(g(t)\big)$ if there exist real numbers $\eta>0$ and $K >0$ such that ${K^{-1}}g(t) \leq f(t) \leq K g(t)$ for all $t\geq0$ satisfying $f(t)\leq \eta$. 

Let $\pi\colon X_\pi \to X$ be a good resolution of $(X,0)$ and let $E_v$ be an irreducible component of $\pi^{-1}(0)$.
Then the \emph{inner rate} $q_v$ of $E_v$ is defined as the inner contact $q_\mathrm{inn}(\gamma,\gamma')$, where $\gamma$ and $\gamma'$ are are two curve germs on $(X,0)$ that pullback via $\pi$ to two curvettes at distinct points of $E_v$.
This definition only depends on the exceptional component $E_v$ and not on the choice of a good resolution on which the component appears (see \cite[Lemma~3.2]{BelottodaSilvaFantiniPichon2019}).

We now recall a deep result, the so-called \emph{Laplacian formula} for the inner rate function from \cite{BelottodaSilvaFantiniPichon2019}.
In order to state it we will introduce two additional vectors indexed by the vertices of the dual graph $\Gamma_\pi$ of a good resolution $\pi\colon X_\pi\to X$ of $(X,0)$.
For every vertex $v$ of $\Gamma_\pi$, set $k_v=\val_{\Gamma_\pi}(v)+2g(v)-2$ and $a_v=m_vq_v$, and consider the vectors $K_\pi=(k_v)_{v\in V(\Gamma_\pi)}$ and $A_\pi=(a_v)_{v\in V(\Gamma_\pi)}$. Let $L_\pi$ and $P_\pi$ be respectively the $\cal L$- and the $\cal P$-vector of $(X,0)$ as before. 
Then the following equality holds:
\begin{equation}\label{equation:laplacian_formula_effective}
A_\pi = I_{\Gamma}^{-1}\cdot(K + L - P_\pi) \,.
\end{equation}
This equality is an effective version (see \cite[Proposition~5.3]{BelottodaSilvaFantiniPichon2019}) of the main result of \emph{loc.\ cit.}
Observe that in our situation $I_{\Gamma}$ and $K + L$ are known, while $A_\pi$ and $P_\pi$ are not, but either of the two is determined by the other one thanks to the formula above. 

In what follows, we argue that in general there is only a very limited number of possible values of $A_\pi$, therefore restricting the number of possible configurations of $P_\pi$.
Since the vector $P_\pi$ has positive coordinates, then $-I_{\Gamma}^{-1} \cdot P_\pi$ is an element of the Lipman cone $\cal E^+$ of $\Gamma$, and so $A_\pi$ belongs to the translate $I_{\Gamma}^{-1}\cdot(K + L) + \cal E^+$ of $\cal E^+$.
Moreover, if a vertex $v_0$ of $\Gamma$ is an $\cal L$-node of $(X,0)$, that is if $l_{v_0}\neq0$, we know that its inner rate $q_{v_0}$ must be equal to 1 (see the paragraph preceding Proposition~3.9 of \cite{BelottodaSilvaFantiniPichon2019}).
This implies that $A_\pi$ belongs to the intersection for $v_0$ running over the set of $\cal L$-nodes of $(X,0)$ of the hyperplanes of $\Z^{V(\Gamma)}$ whose $v_0$-th coordinate is equal to $m_{v_0}$.
Since $D>0$ for every nonzero element $D$ of the Lipman cone $\cal E^+$, this intersection is finite, 
and in fact rather small.
Therefore, the vector $P_\pi$ can only take fewer values than those allowed by the proof of Theorem~\ref{thm:main}.
This construction is illustrated in Figure~\ref{figure:cones} below.

\begin{figure}[ht]
	\centering
	\begin{tikzpicture}[scale=1]
	\draw[thick,->] (-3.5,0) -- (6.5,0);
	\draw[thick,->] (1,-1.35) -- (1,5.5);
	
	\draw[very thick,color=blue] (1,0) -- (3.65,5.3);
	\draw[very thick,color=blue] (1,0) -- (6,1.666);
	
	\draw[very thick,color=red] (-3,-1) -- (7,2.333);
	\draw[very thick,color=red] (-3,-1) -- (.15,5.3);

	\draw[very thick,color=ForestGreen,dashed] (-3.5,2) -- (1,2);
	\draw[ultra thick,color=ForestGreen] (1,2) -- (6,2);
	\draw[very thick,color=ForestGreen,dashed] (5.9,2) -- (7,2);
	
	\foreach \x in {-3,-3,-2,-1,0,1,2,3,4,5,6}
	\foreach \y in {-1,0,1,2,3,4,5}
	\draw[fill ] (\x,\y)circle(1pt);
	
	\draw[fill,color=blue] (1,0)circle(1.7pt);
	\draw[fill,color=blue] (2,1)circle(1.7pt);
	
	\draw[fill,color=red] (-3,-1)circle(1.7pt);
	
	\draw[fill] (1,2)circle(1.7pt);
	
	\draw[fill,color=ForestGreen] (3,2)circle(1.7pt);

	\begin{small}
	\node(a)at(3.7,4.5){{\color{blue}$\mathcal E^+$}};
	\node(a)at(-.6,.6){{\color{red}$I_{\Gamma}^{-1}\cdot(K + L)+\mathcal E^+$}};
	\node(a)at(-2.4,4.5){$\mathbb Z^{V(\Gamma)}$};
	\end{small}
	
	\draw[thick,color=ForestGreen,->] (4.6,3.2) -- (4.4,2.1);
	
	\begin{scriptsize}
	\node(a)at(-2.6,-1.3){{\color{red}$I_{\Gamma}^{-1}\cdot(K + L)$}};
	\node(a)at(2.3,1.2){{\color{blue}$Z_{\min}^{\Gamma}$}};
	\node(a)at(4.7,3.4){{\color{ForestGreen}possible values for $A_\pi$}};
	\node(a)at(3,2.3){{\color{ForestGreen}$A_\pi$}};
	\node(a)at(.48,1.7){$(m_{v_0},0)$};
	\end{scriptsize}
	\end{tikzpicture}
	\caption{	Observe that, since $Z_{\min}^{\Gamma}\succ 0$, then the Lipman cone $\cal E^+$ (in blue), and thus $I_{\Gamma}^{-1}\cdot(K + L)+\mathcal E^+$ (in red),  contain no horizontal line. Only six values of $A_\pi$ are possible in this example.}
	\label{figure:cones}
\end{figure}
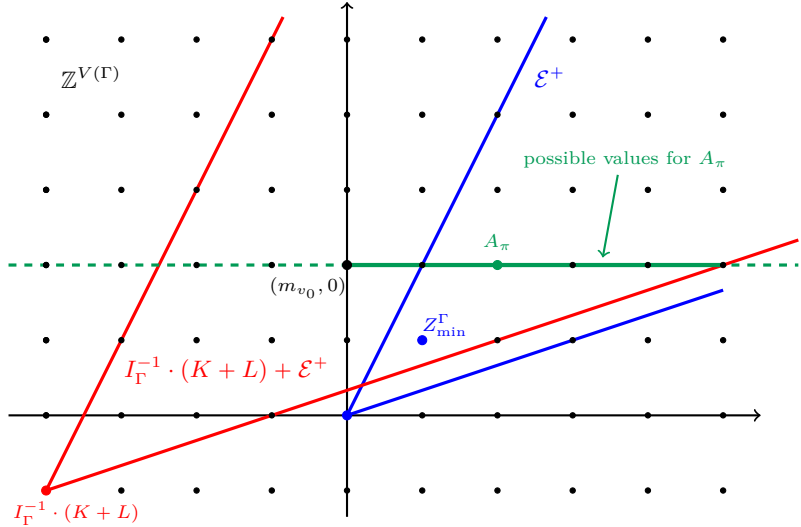

\subsection{Topological constraints}
\label{subsec:topological_constraints}
Additional restrictions may be derived from the topological properties of the germ $(X,0)$, and more specifically from the \emph{local degrees} associated with a generic projection $\ell\colon(X,0)\to(\C^2,0)$. Let $\pi\colon X_\pi\to X$ be a good resolution of $(X,0)$ and let $\ell\colon(X,0)\to(\C^2,0)$ be a generic projection.
Let $\sigma_{\ell} \colon Y_{\ell} \to \C^2$ be a sequence of point blowups of $(\C^2,0)$ such that the rational map $\sigma_\ell^{-1}\circ\ell\circ\pi$ maps each component of $\pi^{-1}(0)$ surjectively to a component of $\sigma_\ell^{-1}(0)$ (that is, no $E_v$ is contracted), so that $\ell$ induces a map $\tilde\ell\colon \Gamma_\pi\to \Gamma_{\sigma_\ell}$ between (the topological spaces underlying) the graphs $\Gamma_\pi$ and $\Gamma_{\sigma_\ell}$.
Let ${\pi_\ell} \colon X_{\pi_\ell} \to X$ be a good resolution of $(X,0)$ such that $\pi_\ell\circ \ell$ factors through $\sigma_\ell$, let $v$ be a vertex of $\Gamma_{\pi}$ and let $v_1, \ldots, v_r$ be the vertices of $\Gamma_{\pi_\ell}$ that are adjacent to $v$ and contained in $\Gamma_\pi$.
Let $\Gamma_v$ be the subgraph of $\Gamma_{\pi_\ell}$ defined as the closure in $\Gamma_{\pi_\ell}$ of the connected component of  $\Gamma_{\pi_\ell}\setminus\widetilde\ell^{-1}\big(\widetilde\ell(\{v_1,\ldots,v_r\})\big)$ containing $v$, and consider the subgraph of $\Gamma_{\sigma_\ell}$ defined as $T_v=\widetilde\ell(\Gamma_v)$. 
Set
\[
\cal N(\Gamma_v) = \overline{ \bigcup_{w \in V(\Gamma_v)} N(E_w) \setminus    \bigcup_{w' \in V(\Gamma_{\pi_\ell}) \setminus V( \Gamma_v)} N(E_{w'})}
\]
and
\[
\cal N(T_v) =  \overline{\bigcup_{w \in V(T_v)} N(E_w) \setminus   \bigcup_{w' \in V(\Gamma_{\sigma_\ell}) \setminus V( T_v)} N(E_{w'})}.
\] 
Adjusting the fiber bundles $N(E_w)$ if necessary, by restricting $\ell$ to $\pi_\ell\big(\cal N(\Gamma_v)\big)$ we obtain a cover 
\(
\ell|_{\pi_\ell(  \cal N(\Gamma_v))}  \colon   \pi_\ell\big(  \cal N(\Gamma_v)\big) \to \sigma_{\ell}\big(\cal N(T_v)\big).
\)
Following \cite[Definition~4.16]{BelottodaSilvaFantiniPichon2019}, we call \emph{local degree of $\ell$ at $v$} the degree of this cover, and we denote it by $\deg(v)$. Pick now a generic linear form $h\colon (X,0)\to (\C,0)$ on $(X,0)$, denote by $\widehat{F}_v$ the intersection of $\pi_\ell \big(\cal N(\Gamma_{v})\big)$ with the Milnor fiber $X\cap\{h=t\}$ of $x$, and set $\widehat{F}'_v=\ell_v(\widehat{F}_v)$.
Restricting again $\ell$, we get a $\deg(v)$-cover   $\ell|_{\widehat{F}_v}\colon \widehat{F}_v\to \widehat{F}_v'$.
The Hurwitz formula applied to this cover yields the following equality:
\begin{lemma}\label{lem:hurwitz}
	\(
	\chi(\widehat{F}_v)+m_vp_v = \deg(v) \chi(\widehat{F}'_v).
	\)
\end{lemma}

\begin{remark}
	It is worth pointing out that the map $\tilde{\ell}$ and the local degree $\deg(v)$ can be defined more intrinsically, without the need of choosing a modification of $(\C^2,0)$, by working with suitable valuation spaces.
	We refer the reader to \cite{BelottodaSilvaFantiniPichon2019}, and in particular to sections 2.1, 2.2, and 4.6 of \emph{loc.\ cit.}, for a more thorough discussion.
\end{remark}

\section{An example of polar exploration}

\begin{example}\label{ex:MaugendreMichel2017}
	We will discuss in detail Example 3 from the paper \cite{MaugendreMichel2017}, showing that we can determine its $\cal P$-vector completely.
	Consider the hypersurface $(X,0)$ in $(\C^3,0)$ defined by the equation  $z^2=(y+x^3)(y+x^2)(x^{34}-y^{13})$.
	The dual graph of the minimal resolution of $(X,0)$ which factors through the blowup of its maximal ideal is given in Figure~\ref{exMM_figure 1}. 
	All exceptional curves are rational, the arrow represents the strict transform of a generic linear form, and the negative numbers are the self intersections of the irreducible components of $\pi^{-1}(0)$. 
	The multiplicities $m_v$, which are computed from this data using the equalities~\eqref{eq:IdentityTotalTransform} (page \pageref{eq:IdentityTotalTransform}), are within parentheses in the figure.
	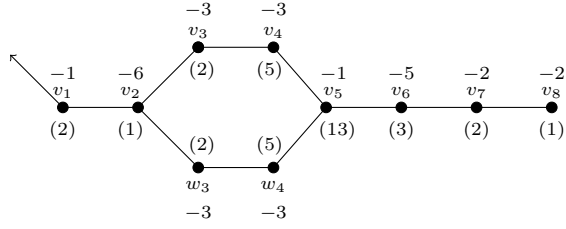
\begin{figure}[ht] 
		\centering
		\begin{tikzpicture}
		
		\draw[fill ] (0,0)circle(2pt);
		\draw[fill ] (1,0)circle(2pt);
		\draw[fill ] (1.8,0.8)circle(2pt);
		\draw[fill ] (1.8,-0.8)circle(2pt);
		\draw[fill ] (1.8,0.8)circle(2pt);
		\draw[fill ] (2.8,-0.8)circle(2pt);
		\draw[fill ] (2.8,0.8)circle(2pt);
		\draw[fill ] (3.5,0)circle(2pt);
		\draw[fill ] (4.5,0)circle(2pt);
		\draw[fill ] (5.5,0)circle(2pt);
		\draw[fill ] (6.5,0)circle(2pt);
		
		\draw[thin ](0,0)--(1,0);
		\draw[thin ](1,0)--(1.8,0.8);
		\draw[thin ](1.8,0.8)--(2.8,0.8);
		\draw[thin ](2.8,0.8)--(3.5,0);
		\draw[thin ](1,0)--(1.8,-0.8);
		\draw[thin ](1.8,-0.8)--(2.8,-0.8);
		\draw[thin ](2.8,-0.8)--(3.5,0);
		\draw[thin ](3.5,0)--(6.5,0);

		\draw[thin,>-stealth,->](0,0)--+(-0.7,0.7);
		
		\begin{scriptsize}
		\node(a)at(0,-0.3){$(2)$};
		\node(a)at(.9,-0.3){$(1)$};
		
		\node(a)at(1.85,.5){$(2)$};
		\node(a)at(2.75,.5){$(5)$};
		\node(a)at(1.85,-.5){$(2)$};
		\node(a)at(2.75,-.5){$(5)$};
		\node(a)at(3.65,-0.3){$(13)$};
		\node(a)at(4.5,-0.3){$(3)$};
		\node(a)at(5.5,-0.3){$(2)$};
		\node(a)at(6.5,-0.3){$(1)$};

		\node(a)at(0,0.45){$-1$};
		\node(a)at(.9,0.45){$-6$};
		\node(a)at(1.8,1.3){$-3$};
		\node(a)at(2.8,1.3){$-3$};
		\node(a)at(1.8,-1.4){$-3$};
		\node(a)at(2.8,-1.4){$-3$};
		\node(a)at(3.6,0.45){$-1$};
		\node(a)at(4.5,0.45){$-5$};
		\node(a)at(5.5,0.45){$-2$};
		\node(a)at(6.5,0.45){$-2$};
		
		\node(a)at(0,0.2){$v_1$};
		\node(a)at(.9,0.2){$v_2$};
		\node(a)at(1.8,1){$v_3$};
		\node(a)at(2.8,1){$v_4$};
		\node(a)at(1.8,-1.05){$w_3$};
		\node(a)at(2.8,-1.05){$w_4$};
		\node(a)at(3.6,0.2){$v_5$};
		\node(a)at(4.5,0.2){$v_6$};
		\node(a)at(5.5,0.2){$v_7$};
		\node(a)at(6.5,0.2){$v_8$};
		\end{scriptsize}	
		\end{tikzpicture}
		\caption{The graph $\Gamma_\pi$, decorated with the self-intersections and multiplicities.}\label{exMM_figure 1}
	\end{figure}	   
	
	Let us now focus on determining the $\cal P$-vector of $(X,0)$ from the data we have available.	Observe that in the following discussion we will not derive any information from the analytic type of the singularity $(X,0)$, but only from the weighted graph $\Gamma_\pi$ and from the multiplicities of its vertices, so that our conclusion will hold for any singularity $(S,0)$ realizing the same data.
	This means in particular that the resolution of $(S,0)$ whose weighted dual graph is $\Gamma_\pi$ will be required to factor through the blowup of the maximal ideal of $(S,0)$, so that the multiplicity $m(S,0)$ of $(S,0)$ is the sum of the products $m_v l_v$ over the vertices $v$ of $\Gamma_\pi$, which in this case is equal to $2$.
	
	Applying \cite[Proposition~3.9]{BelottodaSilvaFantiniPichon2019} to the vertex $v_8$, we deduce that the inner rates are strictly increasing along a path from $v_1$ to $v_8$.
	In particular, if $\widetilde\ell$ is the graph morphism of subsection~\ref{subsec:topological_constraints}, one of the two chains going from $v_2$ to $v_5$ must be sent by $\widetilde\ell$ surjectively onto a chain $\gamma$ in the tree $\widetilde\ell(\Gamma_\pi)$.
	It follows that the image through $\widetilde\ell$ of the second chain from $v_2$ to $v_5$ should contain $\gamma$ as well, and since the multiplicity of the surface is $2$, this shows that the local degree on each of them cannot exceed $1$, so that it must be equal to 1, and therefore the image of both must be precisely $\gamma$.
	In particular the two strings are also strings of the graph $\Gamma_{\pi_\ell}$, where $\pi_\ell$ is the resolution of $(X,0)$ introduced in the discussion of subsection~\ref{subsec:topological_constraints}.
	Then, the connected components of $\Gamma_{\pi} \setminus \{v_5\}$ and $\Gamma_{\pi_\ell} \setminus \{v_5\}$ containing $v_1$ coincide, and thus we can determine $p_v$ for each vertex $v$ of this subgraph by applying Lemma~\eqref{lem:hurwitz}.
	We obtain $p_{v_1} = p_{v_3} =  p_{v_4} = p_{w_3}  = p_{w_4}  =  0$ and $p_{v_2}=1$.
	
	Since the inner rate function achieve its maximum strictly on the vertex $v_8$, then $v_8$ keeps valency one in the graph $\Gamma_{\pi_\ell}$, and applying again Lemma~\eqref{lem:hurwitz} we obtain $\chi(F_{v_8})+m_vp_{v_8} = \deg(v_8) \chi(F'_{v_8})$. 
	Since $\chi(F_{v_8})=1$, we obtain that  $\chi(F'_{v_8}) \geq 1$, so   that $\chi(F'_{v_8}) =1$ since $F'_{v_8}$ has one boundary component. 
	Therefore, $p_{v_8} =   \deg(v_8) -1$, so $p_{v_8} \in \{0,1\}$ since $ \deg(v_8) $ cannot exceed the multiplicity of $(S,0)$, which is $2$. 
	
	Notice that we do not know at this point whether $\Gamma_{\pi_\ell}$ has or not an extra edge adjacent to one of the vertices $v_5$, $v_6$, or $v_7$ whose image through $\widetilde\ell$ is contained in $\widetilde\ell(\Gamma_\pi)$. 
	Applying again Lemma~\eqref{lem:hurwitz}, we then have four cases: $(p_{v_5}, p_{v_6}, p_{v_7}, p_{v_8}) \in \{  (1,0,0,1), (1,0, 1,0), (1,1, 0,0),  (2,0,0,0)  \}$, with the first case corresponding to $\Gamma_{\pi} =\Gamma_{\pi_\ell}$.

	We now use the Laplacian formula recalled in equation~\eqref{equation:laplacian_formula_effective} (page \pageref{equation:laplacian_formula_effective}) to eliminate some of these possibilities for $P$ and to compute the inner rates. 
	Writing the formula for every vertex $v \in \{ v_1, v_2, v_3,v_4 , w_3,w_4\}$, for which we know $p_v$, and using the fact that $q_{v_1}=1$, we obtain the corresponding inner rates $q_v$ and the inner rate $q_{v_5}$, which are as follows: $q_{v_2} = 2,  q_{v_3} =q_{w_3} = \frac{5}{2}, q_{v_4} =q_{w_4} = \frac{13}{5}$, and $q_{v_5} =   \frac{34}{13}$. 
	
	Now, the Laplacian formula for $v_5$ yields $-13q_{v_5} + 5(q_{v_4} + q_{w_4}) + 3 q_{v_6} = 1-p_{v_5}$.
	Therefore $3 q_{v_6} + p_{v_5} = 9$, where $q_{v_6} \in \frac{1}{3} \N$ and $q_{v_6}  > q_{v_5} =   \frac{34}{13}$. 
	Therefore   $q_{v_6}=3$ and $p_{v_5} = 0 $ or $q_{v_6}=\frac{8}{3}$ and $p_{v_5} = 1$. 
	But we know from the Hurwitz arguments above that $p_{v_5} = 1$ or $2$. 
	Therefore, the only possibility is $q_{v_6}=\frac{8}{3}$ and $p_{v_5} = 1$. 
	
	The Laplacian formula for $v_6$ gives $ -15q_{v_6} + 13 q_{v_5} +2q_{v_7} = -p_{v_6}$. 
	Then $2q_{v_7} +p_{v_6}  =11$ with $q_{v_7} \in \frac{1}{2} \N,    q_{v_7}  > \frac{8}{3}$ and  and $p_{v_6} \leq 1$.
	The unique possibility is $q_{v_7} = 3$ and $p_{v_6} \leq 0$. 
	
	The Laplacian formula for $v_7$ gives $ q_{v_8} +p_{v_7}  =4 q_{v_7}-3 q_{v_6} = 4  $, with $q_{v_8} \geq0$, $q_{v_8}  > 3$, and $p_{v_7} \leq 1$.
	The unique possibility is $q_{v_8} = 4$ and $p_{v_7} \leq 0$. 
	
	Among the four possibilities for $(p_{v_5}, p_{v_6}, p_{v_7}, p_{v_8})$ identified above, the unique  possibility is then  $(1,0,0,1)$, so $p_{v_8}=1$.
	This is indeed compatible with the Laplacian formula for $v_8$. 
	
	We have obtained a unique possibility for $P$ and the inner rates, as shown in Figure~\ref{figure 3}. 
	Observe that since there are no edges joining two vertices with nonzero $p_{v}$, then the strict transform $\Pi^*$ of the polar curve  $\Pi$ by $\pi$ meets $E$ at smooth points of the exceptional divisor $\pi^{-1}(0)$. 
	Moreover, since each $p_{v_i}$ equals either zero or one, then $\pi$ is a good resolution of $\Pi$, that is $\pi^{-1}(0)$ is a simple normal crossing divisor.  
	The arrows in Figure~\ref{figure 3} represent the strict transform of a generic polar curve.
	
	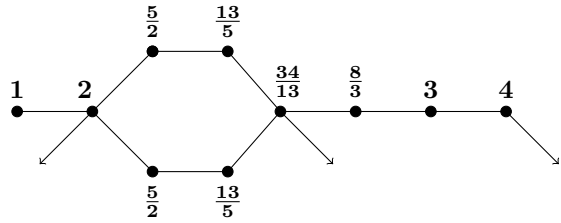
\begin{figure}[ht] 
		\centering
		\begin{tikzpicture}

		\draw[thin,>-stealth,->](1,0)--+(-0.7,-0.7);
		\draw[thin,>-stealth,->](3.5,0)--+(+0.7,-0.7);
		\draw[thin,>-stealth,->](6.5,0)--+(+0.7,-0.7);

		\draw[fill ] (0,0)circle(2pt);
		\draw[fill ] (1,0)circle(2pt);
		\draw[fill ] (1.8,0.8)circle(2pt);
		\draw[fill ] (1.8,-0.8)circle(2pt);
		\draw[fill ] (1.8,0.8)circle(2pt);
		\draw[fill ] (2.8,-0.8)circle(2pt);
		\draw[fill ] (2.8,0.8)circle(2pt);
		\draw[fill ] (3.5,0)circle(2pt);
		\draw[fill ] (4.5,0)circle(2pt);
		\draw[fill ] (5.5,0)circle(2pt);
		\draw[fill ] (6.5,0)circle(2pt);
		
		\draw[thin ](0,0)--(1,0);
		\draw[thin ](1,0)--(1.8,0.8);
		\draw[thin ](1.8,0.8)--(2.8,0.8);
		\draw[thin ](2.8,0.8)--(3.5,0);
		\draw[thin ](1,0)--(1.8,-0.8);
		\draw[thin ](1.8,-0.8)--(2.8,-0.8);
		\draw[thin ](2.8,-0.8)--(3.5,0);
		\draw[thin ](3.5,0)--(6.5,0);
		
		\node(a)at(0,0.3){$\mathbf  1$};
		\node(a)at(.9,0.3){$\mathbf  2$};
		\node(a)at(1.8,1.2){$ \frac{\mathbf  5}{\mathbf  2}$};
		\node(a)at(2.8,1.2){$\frac{\mathbf{13}}{\mathbf  5}$};
		\node(a)at(1.8,-1.2){$ \frac{\mathbf  5}{\mathbf  2}$};
		\node(a)at(2.8,-1.2){$  \frac{\mathbf{13}}{\mathbf  5}$};
		\node(a)at(3.6,0.4){$  \frac{\mathbf{34}}{\mathbf{13}}$};
		\node(a)at(4.5,0.4){$ \frac{\mathbf{8}}{\mathbf{3}}$};
		\node(a)at(5.5,0.3){$\mathbf{3}$};
		\node(a)at(6.5,0.3){$\mathbf{4}$};
		
		\end{tikzpicture}
		\caption{The graph $\Gamma_\pi$, decorated with the inner rates of its vertices and arrows corresponding to the components of a generic polar curve.}\label{figure 3}
	\end{figure}
\end{example}

\begin{remark}
	Observe that at this stage in the previous example it is not possible to know whether the $\cal P$-nodes of $(S,0)$ are $v_2,v_5$, and $v_8$, since in principle the resolution whose weighted dual graph is $\Gamma_\pi$ might not factor through the Nash transform of $(S,0)$.
	For example, in the case of the hypersurface $(X,0)$ we started with, none of those three vertices is a $\cal P$-node.
\end{remark}

\bibliographystyle{alpha}                              
\bibliography{biblio}

\vfill

\end{document}